%% file: main.tex
\documentclass[11pt,a4paper]{amsart}
\usepackage[utf8]{inputenc}
\usepackage[T1]{fontenc}

\pdfoutput=1

\title[Transitive graphs with uniformly bisecting quasi-geodesics]{Vertex-transitive graphs with \\ uniformly bisecting quasi-geodesics}
\author{Joseph Paul MacManus}
\date{First draft: 13th November 2025. This version: 25th August 2026}

\address{School of Mathematics, University of Bristol, Bristol, BS8 1UG, UK, and the Heilbronn Institute for Mathematical Research, Bristol, UK.}
\email{macmanus.maths@gmail.com}

\usepackage{amsfonts}
\usepackage{amsmath}
\usepackage{amsthm}
\usepackage{amssymb}
\usepackage{thmtools}
\usepackage{caption}
\usepackage{stmaryrd}
\usepackage{mathrsfs}  
\usepackage{todonotes}
\usepackage{graphicx}
\usepackage{tikz}
\usepackage{quiver}
\usepackage[colorlinks]{hyperref} 
\hypersetup{
    linkcolor=black,
    citecolor=blue,
}
\usepackage{fancyref}
\usepackage{todonotes}

\usepackage[abbrev]{amsrefs}

\input{commands}

\begin{document}

\begin{abstract}
    Suppose that $X$ is an infinite, connected, locally finite, quasi-transitive graph with the property that every bi-infinite quasi-geodesic uniformly coarsely separates $X$ into exactly two deep pieces. We show that such an $X$ is quasi-isometric to either the Euclidean plane or the hyperbolic plane. In particular, if $X$ is a Cayley graph of a finitely generated group $G$ with the above property, then $G$ is a virtual surface group. This can be interpreted as an extension of the well-known fact that a hyperbolic group with circular boundary is virtually Fuchsian. 
    
    Our theorem positively resolves Problem~14.98 of the Kourovka Notebook, posed by V. A. Churkin in 1999.  
    The proof uses an isoperimetric inequality of Coulhon--Saloff-Coste to show that if such a graph has the above property, then either it is hyperbolic or has quadratic growth.
\end{abstract}

\maketitle





\section{Introduction}

Planarity is known to be among the most rigid properties of finitely generated groups and, more generally, of locally finite, quasi-transitive graphs. Perhaps one of the most important examples of this phenomenon is a theorem rooted in work of Mess on the Seifert conjecture \cite{mess1988seifert}, which states that every finitely generated group that is quasi-isometric to a complete Riemannian plane is a virtual surface group. Mess dealt directly with the non-hyperbolic case using an inequality of Varopoulos to bound the growth of the group, while the hyperbolic case was later completed with the celebrated convergence group theorem of Tukia \cite{tukia1988homeomorphic}, Gabai \cite{gabai1992convergence}, and Casson--Jungreis \cite{casson1994convergence}.
Bowditch later gave a new proof and several remarkable algebraic extensions of this result, making use of a certain kind of `winding number' on the Cayley complex \cite{bowditch2004planar}. Around the same time, another coarse characterisation of virtual surface groups was given by Papasoglu in \cite{papasoglu2005quasi}, which describes virtual surface groups as, roughly speaking, those one-ended finitely presented groups whose Cayley graphs have the property that any two sufficiently far apart points can be separated by a quasi-line.
More recently, virtually planar groups were characterised by the present author among finitely presented groups as those which are asymptotically minor excluded \cite{macmanus2025fat}, and among finitely generated groups as those which are quasi-isometric to a planar graph \cite{macmanus2023accessibility}. An interesting extension of the latter was recently presented by Davies in \cite{davies2025string}. 
 
We now return our attention to the convergence group theorem, mentioned above. This asserts that if a group $G$ admits a discrete uniform convergence action on the circle, then this action is topologically conjugate to one which is induced by some geometric action of $G$ on the hyperbolic plane. In particular, this implies that any hyperbolic group with circular Gromov boundary is virtually Fuchsian. 
The goal of this paper is to present an extension of this result applicable to all locally finite, quasi-transitive graphs. Of course, if the graph in question is not hyperbolic then we do not necessarily have a natural or helpful notion of a `boundary at infinity'. Thus, to clarify the situation, it makes sense to reframe our problem in terms of the intrinsic geometry of the Cayley graph.
To this end, we invite the reader to consider the following definition. 

\begin{definition}\label{def:ubq}
    Let $X$ be a geodesic metric space. A connected subspace $Z \subset X$ is called \emph{narrow} if it is quasi-isometric to a subset of $\R$, with respect to the ambient metric. If $Z$ is not narrow then it is said to be \emph{wide}. We define the \emph{uniformly bisecting quasi-geodesics} property, denoted (\ref{eq:BQ}), as follows:
        \begin{equation}
  \tag{$\rm{UBQ}$}\label{eq:BQ}
  \parbox{4in}{%
    \strut
    There exists $\sigma > 0$ such that for any connected, bi-infinite quasi-geodesic $\rho \subset X$, we have that $X \setminus N_\sigma(\rho)$ contains exactly two wide connected components. 
    \strut
  }
\end{equation}
We call $\sigma$ the \emph{separation constant}.
\end{definition}

The property \property{} is easily checked to be a quasi-isometry invariant among geodesic spaces. 
Thus, it makes sense to say that a finitely generated group has this property. The Euclidean and hyperbolic planes obviously have \property{}, by the Jordan curve theorem. It is an exercise in hyperbolic geometry to show that if an infinite hyperbolic group satisfies \property{} then its Gromov boundary is circular. Hence, the convergence group theorem implies the following.

\begin{theorem*}[Corollary of the convergence group theorem]
    Let $G$ be an infinite hyperbolic group with \property{}. Then $G$ is virtually Fuchsian.
\end{theorem*}

In this form, the hypotheses now make sense in any finitely generated group, or indeed in any locally finite, quasi-transitive graph.
It is therefore natural to ask whether this theorem extends to non-hyperbolic groups or to arbitrary locally finite, quasi-transitive graphs. This question appears as Problem~14.98 of the Kourovka Notebook \cite{khukhro2025unsolvedproblemsgrouptheory}, posed by V. A. Churkin in 1999. 
In this paper, we answer Churkin's question in the affirmative and prove the following.

\begin{restatable}{alphtheorem}{main}\label{thm:main-theorem}
    Let $X$ be an infinite, connected, locally finite, quasi-transitive graph with \property{}. Then $X$ is quasi-isometric to either the Euclidean plane or the hyperbolic plane.  
\end{restatable}

Since finitely generated abelian groups and Fuchsian groups are known to be quasi-isometrically rigid, this has the consequence that if $G$ is a finitely generated group with \property{} then $G$ is a virtual surface group. Thus we provide a new, purely geometric characterisation of virtual surface groups among all finitely generated groups.

We remark that Theorem~\ref{thm:main-theorem} can be compared to Bing's characterisation of the 2-sphere \cite{bing1946kline}. This states that any compact, connected, locally connected space without cut-points or cut-pairs, in which every simple closed curve separates, is homeomorphic to the 2-sphere. Morally, our theorem is a coarse analogue of this statement for locally finite, quasi-transitive graphs.
We also point out that our characterisation is similar in spirit to Papasoglu's theorem mentioned above, in the sense that they both characterise virtual surface groups through the separation properties of their quasi-lines. One key difference between the two results is that Papasoglu requires the group to be \emph{a priori} finitely presented, whereas we do away with this assumption altogether.

\subsection*{Outline}

Before we begin, we describe a rough outline of this paper and the proof therein.
In order to prove Theorem~\ref{thm:main-theorem}, we consider separately the cases where $X$ is hyperbolic and where $X$ is not hyperbolic. The hyperbolic case follows from the convergence group theorem as discussed above.\footnote{The case where $X$ is hyperbolic but not a Cayley graph requires some care, but still follows from known results; see~\cite[\S1.1]{macmanus2024note}.} Our attention is thus entirely focused on the case where $X$ is not hyperbolic.
The ultimate goal is to prove that non-hyperbolicity plus \property{} implies that $X$ has quadratic growth. From here, a theorem of Trofimov  \cite{trofimov1985graphs}, building on work of Gromov \cite{gromov1981groups}, plus the Bass--Guivarc'h formula for the degrees of growth of nilpotent groups \cite{bass1972degree, guivarc1973croissance}, allows us to conclude that $X$ is quasi-isometric to the Euclidean plane.

We will study this problem through the eyes of \emph{witnesses}, which are introduced in \S\ref{sec:witness}. These are pairs of rays which probe deep into the two wide components bound by a bi-infinite quasi-geodesic. A key tool here will be the \emph{witness protection lemma}, which asserts that if we perturb a bi-infinite geodesic in a bounded way, then the original witnesses still witness two distinct deep components bound by the new quasi-geodesic. Following this, in \S\ref{sec:crosshair} and \S\ref{sec:witness-sep} we construct and consider a large figure called a \emph{tripod configuration}. Attaching paths to the tripod configuration and applying the witness protection lemma allows us to construct closed loops which coarsely bound regions of $X$ in a controlled way. These loops should be compared to Jordan curves in the plane. Combining this with an inequality of Coulhon--Saloff-Coste \cite{CoulhonSaloffCoste1993,saloff1995isoperimetric}, 
in \S\ref{sec:complementary depth} we establish a useful criterion for showing that $X$ has quadratic growth. 

Following this, we study the separation properties of \emph{quasi-circles} and, more generally, \emph{truncated quasi-circles} in \S\ref{sec:circles}. We apply our new criterion for quadratic growth and see that either  all (truncated) quasi-circles coarsely bound regions of at most uniformly bounded depth, regardless of length, or $X$ has quadratic growth. In the language of this paper, we prove that truncated quasi-circles have \emph{uniformly bounded complementary depth}. In the final section, \S\ref{sec:final}, we apply the Arzel\`a--Ascoli theorem together with a result of Hume--Mackay \cite{hume2020poorly} to construct a sequence of `nested' truncated quasi-circles lying along a bi-infinite quasi-geodesic. We then observe that this `nesting' taken together with the witness protection lemma means the truncated quasi-circles cannot possibly have uniformly bounded complementary depth. This enables us to conclude that $X$ has quadratic growth, and complete the proof.

\subsection*{Acknowledgements}

 I thank Panos Papasoglu for helpful discussions about bigons, and Agelos Georgakopoulos for comments. This work was supported by the Heilbronn Institute for Mathematical Research.


\section{Preliminaries}\label{sec:prelims}

In this paper, every graph is assumed to be simplicial. We will often conflate a graph with its `geometric realisation', and view a graph as a CW-complex, metrised in the obvious way with edges having unit length. Given a graph $X$, let $\dist_X$ denote this metric. We write $V(X)$ for its vertex set and $E(X)$ for its set of (unoriented) edges, respectively. 

For us, a path $p$ is a sequence of vertices $p = v_0, v_1, \ldots, v_n$ such that $v_i$ and $v_{i+1}$ are connected by an edge for every $0 \leq i < n$. The vertices $v_0$ and $v_n$ are called the \emph{initial} and \emph{terminal} vertices of $p$. We also allow bi-infinite and one-way infinite paths. Given a finite path $p$ as above, we define $\length(p) := n$. The path $p$ is said to be \emph{simple} if every $v_i$ appears exactly once. A finite path is called a \emph{cycle} or \emph{loop} if $v_0 = v_n$, and is said to be a \emph{simple loop} if every $v_i$ appears exactly once, except for the endpoints which appear twice. 
We will often abuse notation and parametrise paths as maps $p : I \to X$. Paths will \textbf{always} be parametrised at unit speed. Given two paths $p$, $q$ such that the terminal vertex of $p$ is equal to the initial vertex of $q$, we write $p \cdot q$ to mean their concatenation. We also write $p^{-1}$ to mean the reversal of the path $p$. 
Let $\lambda \geq 1$, $c \geq 0$. Given a path $p : I \to X$, we call $p$ a \emph{$(\lambda, c)$-quasi-geodesic} if
$$
\dist_X(p(s),p(t)) \geq \frac 1 \lambda |s-t| - c,
$$
for all $s,t \in I$. 
Throughout the rest of this paper, we will adhere to the following conventions.

\begin{convention}
    Unless otherwise stated, we fix $X$ and $\sigma > 0$ throughout this paper where $X$ is an  infinite, connected, locally finite, vertex-transitive graph which satisfies \property{} with separation constant $\sigma$. 
\end{convention}

\begin{convention}
    When we talk of a quasi-geodesic, we will always mean a connected path parametrised at unit speed. 
\end{convention}

\section{Witness protection}\label{sec:witness}

The purpose of this section is to motivate and prove a key lemma, which we call the \emph{witness protection lemma}. This is Lemma~\ref{lem:witness-protection} below. Before this, we need to make a few basic observations.

\begin{proposition}\label{prop:one-ended}
    $X$ is one-ended. 
\end{proposition}

\begin{proof}
    Note we assumed $X$ to be infinite, and clearly $X$ cannot be two-ended. Suppose $X$ is infinite-ended. In this case, it is easy to find a quasi-geodesic $\rho$ and $R > 0$ such that $X \setminus N_R(\rho)$ has infinitely many wide connected components. Indeed, we may simply take $\rho$ to be a geodesic connecting any two distinct ends. It follows that $X$ must be one-ended.
\end{proof}

The following remark is also worth noting. 

\begin{remark}\label{rmk:nbhds-bisect}
    If $\rho$ is a bi-infinite quasi-geodesic, then tubular neighbourhoods of $\rho$ also `bisect' $X$. That is, given $r \geq \sigma$, if we let $Z = N_r(\rho)$, then $X \setminus Z$ also contains exactly two wide connected components. This is because there exists a bi-infinite quasi-geodesic $\rho'$ whose vertex set is exactly $Z$, obtained by travelling along $\rho$, visiting every vertex in the $r$-ball around a given vertex, then travelling to the next vertex along $\rho$, and so on. 
\end{remark}

We now prove a sequence of basic lemmata, the first of which presents a way of recognising wide components. 

\begin{lemma}\label{lem:witness-implies-wide}
    Let $\rho \subset X$ be a bi-infinite quasi-geodesic. Let $U$ be a connected component of $X \setminus N_\sigma(\rho)$. Suppose that there is a one-way infinite simple path $p$ contained in $U$ such that $\lim_{t \to \infty} \dist(p(t), \rho) = \infty$. Then $U$ is wide. 
\end{lemma}

\begin{proof}
    Suppose $U$ is narrow, so that $U$ is quasi-isometric to some connected---and necessarily infinite---subset of $\R$. We therefore have two cases, either $U$ is quasi-isometric to the full-line $\R$, or to the half-line $\R^+$.

    First, let us assume that $U$ is quasi-isometric to the half-line $\R^+$. Then there exists a one-way infinite quasi-geodesic ray $\alpha \subset U$ such that $\dHaus(U, \alpha)$ is finite. We now note that $\alpha$ must diverge from $\rho$. That is,  $\lim_{t \to \infty} \dist(\alpha(t), \rho) = \infty$. Indeed, if this were not the case then $p$ itself could not diverge from $\rho$.
    In particular, this implies that $X$ has multiple ends which contradicts Proposition~\ref{prop:one-ended}. 

    The case where $U$ is quasi-isometric to $\R$ follows identically, as we need only notice that $p$ will spend an infinite amount of time in exactly one `half' of $U$. The corresponding `half' of $\alpha$ will then diverge from $\rho$ and we conclude similarly. 
\end{proof}

The next lemma provides a converse to the above. 

\begin{lemma}\label{lem:wide-implies-witness}
    Let $\rho \subset X$ be a bi-infinite quasi-geodesic, and let $U$ be a wide component of $X \setminus N_\sigma(\rho)$. Then there exists a path $p \subset U$ such that  $\lim_{t \to \infty} \dist_X(p(t), \rho) = \infty$.  
\end{lemma}

\begin{proof}
    Suppose not, then there exists $R > 0$ such that every component of $U \setminus N_R(\rho)$ is finite. But then $X \setminus N_R(\rho)$ has at most one wide component. This contradicts Remark~\ref{rmk:nbhds-bisect}.
\end{proof}

In light of Lemmata~\ref{lem:witness-implies-wide} and \ref{lem:wide-implies-witness}, we introduce the following terminology. 

\begin{definition}[Witnesses] 
    Let $\rho \subset X$ be a bi-infinite quasi-geodesic. We say a pair $(w_1,w_2)$ of one-way infinite simple paths in $X$ is a \emph{pair of witnesses for $\rho$} if the following hold:
    \begin{enumerate}
        \item $w_1$ and $w_2$ are contained in distinct wide components of $X \setminus N_\sigma(\rho)$, and
        \item For both $i = 1,2$, we have $ \dist(w_i(t), \rho) \to \infty$ as $t \to \infty$. 
    \end{enumerate}
\end{definition}

We now have the following key lemma, which will be employed throughout this paper. It says that pairs of witnesses are preserved by bounded perturbations of the quasi-geodesic.

\begin{lemma}[The witness protection lemma]\label{lem:witness-protection}
    Let $\rho_1, \rho_2 \subset X$ be bi-infinite quasi-geodesics such that $\dHaus(\rho_1, \rho_2)$ is finite. Suppose $(w_1, w_2)$ is a pair of witnesses for $\rho_1$. Then there are infinite subpaths $w_1'\subset w_1$ and $w_2' \subset w_2$ such that $(w_1', w_2')$ is a pair of witnesses for $\rho_2$. 
\end{lemma}

\begin{proof}
    Fix a pair of witnesses $(w_1, w_2)$ for $\rho_1$ and let $(u_1, u_2)$ be a pair of witnesses for $\rho_2$. 
    Write $R = \dHaus(\rho_1, \rho_2)$. Let $Z = N_R(\rho_1)$, so $\rho_2 \subset Z$. By \property{} and Remark~\ref{rmk:nbhds-bisect} we have that $X \setminus N_\sigma(Z)$ contains exactly two wide connected components. Trivially, we have that $N_\sigma(\rho_1)$ and $ N_\sigma(\rho_2)$ are contained within $N_\sigma(Z)$, and so 
    $
    X\setminus N_\sigma(Z) \subset X \setminus N_\sigma(\rho_i)
    $
    for both $i =1,2$. 

    Note that the $w_i$ diverge from $\rho_1$, and thus diverge from $Z$. Thus there are infinite subpaths $w_1'$, $w_2'$ of $w_1$ and $w_2$ respectively which are contained in $X \setminus N_\sigma(Z)$. By Lemma~\ref{lem:witness-implies-wide}, we have that each of the $w_i'$ is contained in a wide component of  $X\setminus N_\sigma(Z)$. They are certainly contained in distinct components, and so $(w_1', w_2')$ is a pair of witnesses for $Z$. Similarly, there exists infinite subpaths $u_1'$, $u_2'$ of $u_1$ and $u_2$ respectively such that $(u_1', u_2')$ is also a pair of witnesses for $Z$. 
    
    Recall that $X \setminus N_\sigma(Z)$ contains exactly two wide connected components. In particular, without loss of generality we may assume that $w_1'$ and $u_1'$ are contained in the same wide component of $X\setminus N_\sigma(Z)$, as are $w_2'$ and $u_2'$. 
    Now, $X\setminus N_\sigma(Z) \subset X \setminus N_\sigma(\rho_2)$, and so $w_i'$ is contained in the same component of $X \setminus N_\sigma(\rho_2)$ as $u_i$. It follows now that $(w_1', w_2')$ is a pair of witnesses for $\rho_2$.  
\end{proof}

\begin{remark}
    When applying the witness protection lemma later in this paper, it will quickly become tedious to say that `there exist infinite tails $w_1'$ and $w_2'$ of $w_1$ and $w_2$ respectively such that $(w_1', w_2')$ is a pair of witnesses for $\rho_2$...'. We will thus sometimes abuse terminology and say `...$(w_1, w_2)$ is also a pair of witnesses for $\rho_2$' to mean the above, even if $w_1$ and $w_2$ are not necessarily completely contained in the corresponding component they are witnessing. If there is any risk of ambiguity then we will be more precise in our language. 
\end{remark}

We now present two applications of the witness protection lemma~(\ref{lem:witness-protection}). 

\begin{lemma}\label{lem:witness-intersects-q}
    Consider the following scenario:
    \begin{enumerate}
        \item Let $\rho$ be a simple bi-infinite quasi-geodesic.

        \item Let $U_1, U_2 \subset X\setminus N_\sigma(\rho)$ denote the two wide complementary components.

        \item Let $(w_1, w_2)$ be a pair of witnesses for $\rho$, with $ w_i \subset U_i$. 

        \item Let $p$ be a path connecting $w_1(0)$ to $w_2(0)$, and write $K = N_{10\sigma}(p)$.

        \item Let $\rho_+, \rho_-$ denote the two infinite connected components of $\rho \setminus K$.

        \item Finally, let $q$ be a path connecting $\rho_+$ to $\rho-$ which is disjoint from $K$.
    \end{enumerate}
    Then either $w_1$ or $w_2$ must intersect $N_\sigma(q)$ 
\end{lemma}

\begin{proof}
    Suppose this were not the case. We then modify $\rho$ to produce a new quasi-geodesic $\rho'$ by connecting the tails of $\rho_-$ and $\rho_+$ together with $q$. This is clearly also a bi-infinite quasi-geodesic, and lies a finite Hausdorff distance from $\rho$. Thus, by the witness protection lemma~(\ref{lem:witness-protection}) we must also have that $(w_1', w_2')$ is a pair of witnesses for $\rho'$, for some tails $w_i' \subset w_i$. However, $N_{\sigma}(\rho')$ is disjoint from the union $w_1 \cup p \cup w_2$. In particular, any two tails of $w_1$, $w_2$ lie in the same connected component of $X \setminus N_\sigma(\rho')$. This is a contradiction. See Figure~\ref{fig:witness-protection-first-app} for a cartoon.
\end{proof}

\begin{figure}[ht]
    \centering
    \input{figures/witness-prot-first-app}
    \caption{Cartoon of the proof of Lemma~\ref{lem:witness-intersects-q}.}
    \label{fig:witness-protection-first-app}
\end{figure}

\begin{lemma}\label{lem:lies-in-bounded-component}
    Consider the following scenario:
    \begin{enumerate}
        \item Let $\rho$ be a simple bi-infinite quasi-geodesic.

        \item Let $U_1, U_2 \subset X\setminus N_\sigma(\rho)$ denote the two wide complementary components.

        \item Let $(w_1, w_2)$ be a pair of witnesses for $\rho$, with $ w_i \subset U_i$. 

        \item Let $p$ be a path connecting $w_1(0)$ to $w_2(0)$, and write $K = N_{10\sigma}(p)$.

        \item Let $\rho_+, \rho_-$ denote the two infinite connected components of $\rho \setminus K$.

        \item Let $q_1$ be a path connecting $\rho_+$ to $\rho-$ which is disjoint from $K$ and disjoint from $N_\sigma(w_2)$, and write $K' = N_{10\sigma}(q_1)$. 

        \item Let $\rho_+'$, $\rho_-'$ denote the two infinite connected components of $\rho \setminus K'$.

        \item Let $q_2$ be a path connecting $\rho_+$ to $\rho-$ which is disjoint from $K'$ and disjoint from $N_\sigma(w_2)$.

        \item Finally, let $s \subset \rho$ be a subpath connecting the endpoints of $q_2$, and let $\ell =  s^{-1} 
        \cdot q_2$. 
    \end{enumerate}
    Then the connected component of $w_1\setminus N_\sigma(\ell)$ containing $w_1(0)$ lies in a bounded component of $X \setminus N_\sigma(\ell)$.
\end{lemma}

\begin{proof}
    
Let $\rho_+''$, $\rho_-''$, $w_1'$, $w_2'$ denote the infinite connected components of $\rho_+ \setminus N_{\sigma}(\ell)$, and so on. 
        Suppose this were not true. Then there exists a path $b \subset X$ which avoids $\setminus N_{\sigma}(\ell)$, and connects $w_1(0)$ to either $w_1'$, $w_2'$, or one of the $\rho'_\pm$. Since the $\rho_\pm$ and $w_i$ are all pairwise far away from each other, we can assume that $b$ only passes near the one which it is connecting to, and avoids the $2\sigma$-neighbourhood of the others. 

        Assume first that $b$ connects to $w_1'$. Then, consider the quasi-geodesic $\hat \rho$ formed by replacing the middle segment of $\rho$ with $q_2$. Note that there is a path connecting $w_1(0)$ to $w_2$ which avoids $\hat \rho$. Note also that $q_2$ does not pass near $w_2$ at all. Now, by the witness protection lemma~(\ref{lem:witness-protection}) we must have that $(w_1', w_2)$ is a pair of witnesses for $\hat \rho$. However, the presence of $b$ creates a path from  $w_1'$ to $w_2$ which avoids $\hat \rho$. This is a contradiction, and so in this case we are done. See Figure~\ref{fig:An bounded case1} for a cartoon of this first case. The case where $b$ connects to $w_2'$ follows from a similar argument. 

        \begin{figure}[ht]
            \centering
            \input{figures/An-bounded-case-1}
            \caption{}
            \label{fig:An bounded case1}
        \end{figure}

        Secondly, let us assume that $b$ connects to $\rho_\pm''$. Without loss of generality, let us assume it connects to $\rho_+''$. In this case we modify $\rho$ and form $\hat \rho$ as follows. 
        Let $\hat \rho_-$ denote the subpath of $\rho_-$ connecting the infinite end of $\rho_-$ to $y_1$. 
        Let $w_1''$ be the maximal initial segment of $w_1$ which contains $w_1(0)$ and does not intersect $N_\sigma(q_1)$. Let $c_1$ be a subpath of $q_1$ which connects $y_1$ to a point in the $\sigma$-neighbourhood of the terminal point of $w_1''$. 
        Let $c_2$ be a geodesic of length $\sigma+1$ connecting the endpoints of $c_1$ and $w_1''$ together.  Finally, let $\hat \rho_+$ denote the tail of $\rho_+$ beginning at $x_n$. 
        Let 
        $$
        \hat \rho = \hat \rho_-^{-1} \cdot c_1 \cdot c_2 \cdot w_1''^{-1} \cdot b \cdot \hat \rho_+.
        $$
        Note that $\hat \rho$ avoids the ball $B$. 
        Once again, by the witness protection lemma~(\ref{lem:witness-protection}) we have that $(w_1, w_2)$ must still be a pair of witnesses for this $\hat \rho$, after possibly passing to tails of the $w_i$. However, this cannot be the case. Indeed, let $h_1$ denote a geodesic of length $\sigma+1$ connecting the initial vertex of $w_1'$ to $q_2$. Let $h_2$ denote the subpath of $\ell$ connecting the terminal point of $h_1$ to $u$, which does not pass through $\rho_-$. Now, let $h_3$ be a geodesic connecting $u$ to $w_2(0)$. We have that the bi-infinite path
        $$
        w_1'^{-1} \cdot h_1 \cdot h_2 \cdot h_3 \cdot w_2
        $$
        avoids the $\sigma$-neighbourhood of $\hat \rho$. Thus, it cannot be the case that any pair of tails of the $w_i$ constitute a pair of witnesses for $\hat \rho$. See Figure~\ref{fig:An bounded case2} for a cartoon of the contradiction reached in this second case. 
    \end{proof}

    \begin{figure}[ht]
            \centering
            \input{figures/An-bounded-case-2}
            \caption{}
            \label{fig:An bounded case2}
        \end{figure}

\section{A separating configuration}\label{sec:crosshair}

In this section we construct a certain gadget, called a \emph{tripod configuration}. 
Its utility will lie in its ability to help us build separating loops with controlled interiors.

\begin{definition}[Tripod configuration]
    Let $\lambda \geq 1$, $c \geq 0$, and $R> r > 10\sigma$ be constants. Then an \emph{$(r, R, \lambda, c)$-tripod configuration} is a decuplet 
    $$
    \crosshair = (v_0;\gamma_1, \gamma_2,\gamma_3;w_1,w_2,w_3;q_1,q_2,q_3), 
    $$
    consisting of the following data:
    \begin{enumerate}
        \item A vertex $v_0 \in V(X)$, called the \emph{basepoint}.

        \item Three one-way quasi-geodesic rays $\gamma_1$, $\gamma_2$, $\gamma_3$ based at $v_0$.

        \item Three one-way infinite paths $w_1$, $w_2$, $w_3$ . 

        \item Three finite path segments $q_1$, $q_2$, $q_3$.
    \end{enumerate}
    This data further satisfies the following axioms for each $i \in \Z/3\Z$:
    \begin{enumerate}
        \item[(TC1)]\label{itm:ce1} The concatenation 
        $$
        \rho_{i} := \gamma_{i+1}^{-1} \cdot \gamma_{i+2}
        $$
        is a bi-infinite $(\lambda,c)$-quasi-geodesic.

        \item[(TC2)]\label{ce2} We have that $w_i(0)$ and $v_0$ lie in the same connected component of the complement $\rho_i \setminus N_{10\sigma}(q_1 \cup q_2 \cup q_3)$. 

        \item[(TC3)]\label{ce3} The pair $(w_i, \gamma_i)$ is a pair of witnesses for $\rho_i$.

        \item[(TC4)]\label{ce4} For all $t > r$, we have that $\dist_X(w_i(t), \rho_i) > 10\sigma$. 

        \item[(TC5)]\label{ce5} The path $q_i$ starts on $\gamma_{i+1}$ and ends on $\gamma_{i+2}$, and 
        $$
        q_i \cap N_{10\sigma}(w_{i+1} \cup w_{i+2} \cup \gamma_i) = \emptyset.
        $$

        \item[(TC6)]\label{ce6} We have that  $q_i \subset N_{R-10\sigma}(v_0) \setminus N_{r+10\sigma}(v_0)$. 

    \end{enumerate}
    
\end{definition}

Such an object is best visualised; see Figure~\ref{fig:crosshair} for a cartoon.

\begin{figure}[ht]
    \centering
    \input{figures/crosshair-final}
    \caption{Cartoon of a tripod configuration.}
    \label{fig:crosshair}
\end{figure}

\begin{remark}
    Some remarks are in order.
    \begin{enumerate}
        \item The constant $R$ in the above definition is superfluous and does not impose any additional restrictions. Instead, it should be thought of as simply measuring the `radius' of our tripod configuration. We have chosen to include this constant in our notation for convenience and ease of reference.

        \item Note that there is something which is slightly misleading about this `planar' drawing of a tripod configuration in Figure~\ref{fig:crosshair}. That is, it is not immediately obvious from the definition that the three witnesses $w_i$ pairwise diverge from each other. For now, this does not bother us, but it will later turn out that this is indeed the case. This is the content of the \emph{witness separation lemma}~(\ref{lem:witness-separation}), found below.

        \item From this point on, we will simply write 
        $
        \crosshair = (v_0, \gamma_i,w_i,q_i)
        $
        to ease notation.
    \end{enumerate}
\end{remark}

Our immediate goal is to prove the existence of tripod configurations.

\begin{proposition}\label{prop:crosshair-exists}
     $X$ contains a tripod configuration.
\end{proposition}

\begin{proof}
    We begin by constructing the $\gamma_i$. Firstly, it is easy to see that $X$ contains a bi-infinite geodesic line. This can be easily verified by taking arbitrarily long geodesic segments of even length, translating them to have a common midpoint, and taking the limit using the Arzel\`a--Ascoli theorem. 
    Now, given such a bi-infinite geodesic $\rho$, one can find points arbitrarily far away from it, and attach them to $\rho$ with shortest paths. 
    If we now translate all the `junction points' to a common vertex and pass to a subsequence using the Arzel\`a--Ascoli theorem again, we find another bi-infinite geodesic $\rho'$ and a one-way infinite geodesic $\gamma$ based on $\rho$ such that $\dist_X(\gamma(t), \rho') = t$ for every $t > 0$, if $\gamma$ is parametrised at unit speed. 
    Write $\gamma_1 := \gamma$ and let $\gamma_2$ and $\gamma_3$ denote the tails of $\rho'$ based at $\gamma(0)$. Write $\rho_i := \gamma_{i+1}^{-1} \cdot \gamma_{i+2}$.

    \begin{claim}
        Each $\rho_i$ is a $(4,0)$-quasi-geodesic.
    \end{claim}

    \begin{proof}
        Any initial segment of the third geodesic $\gamma_1$ we construct is easily seen to be a shortest path to the geodesic $\rho_1$. From this, it is an exercise in the triangle inequality that $\rho_2$ and $\rho_3$ are $(4,0)$-quasi-geodesics. 
    \end{proof}


    Now, let $w_1$, $w_2$, $w_3$ be any choice of paths such that $(w_i, \gamma_i)$ is a pair of witnesses for $\rho_i = \gamma_{i+1}^{-1} \cdot \gamma_{i+2}$. Assume without loss of generality that the $w_i$ are disjoint from $\sigma$-neighbourhood of all three of the $\gamma_j$ (this is clearly possible by, say, passing to a tail of $w_i$). Then, extend the initial segment of $w_i$ with a shortest path to $\rho_i$, so that $w_i(0) \in \rho_i$. 

    Since $w_i$ is a witness to $\rho_i$, there exists $r_i > 0$ such that $\dist_X(w_i(t) , \rho_i) > 10\sigma$ for all $t > r_i$. Note that this also implies that $\dist_X(w_i(t), \gamma_i) > 100\sigma$ for all $t > r_i$. Let
    $$
    r := 10\max \{ r_1, r_2, r_3 \} + 10\sigma.
    $$
    We now choose the path segments $q_i$. This will require some extra care and/or trickery.
    Let $B = N_{r+10\sigma}(v_0)$. 
    Let $\gamma_i' = \gamma_i\setminus B$ Note that the $\gamma_i'$ are all connected and pairwise far apart from each other (i.e. they are $10\sigma$-separated), and the $w_i$ are similarly far away from all of the $\gamma_i'$. 
    
    Since $X$ is one-ended by Proposition~\ref{prop:one-ended}, let $q$ be a path from $\gamma_{1}'$ to $\rho_1$ which avoids $B$, of minimal length. Let us assume without loss of generality that $q$ terminates on $\gamma_2'$. 

    \begin{claim}
        The path $q$ is disjoint from $N_{10\sigma}(w_{1} \cup w_2 \cup \gamma_3)$.
    \end{claim}

    \begin{proof}
         Firstly, note that $q$ cannot pass anywhere near $w_1$, as to do so it would need to pass too close to $\rho_1$, since $(\gamma_1, w_1)$ is a pair of witnesses for $\rho_1$, and this would contradict minimality.
        Next, by considering the inverse path $p_2$, we note also that, in this case, $q$ will not pass anywhere near $w_2$ either. Indeed, since $(w_2, \gamma_2)$ is a pair of witnesses for $\rho_2$, if $q$ were to travel near both $w_2$ and $\gamma_2$ it will first need to pass $\sigma$-close to $\rho_2$, i.e. near $\gamma_1$ or $\gamma_3$, thus contradicting minimality of $q$. 
        Finally, we obviously have that $q$ does not travel $10\sigma$-near $\gamma_3$, as if it did so before it reached $\gamma_2$, we would once again find a shorter path, since $\gamma_2'$ and $\gamma_3'$ are comparatively very far apart. 
    \end{proof}

    Repeat the same construction with $\gamma_3$, and construct a path $q'$ with similar properties. If instead the path $q$ actually terminated at $\gamma_3'$, then we repeat the construction instead with $\gamma_2$. 

    We would now like to repeat the construction once more and obtain our third path through the same method. However, this might not work as we do not have any control as to which `side' the constructed path connects to. Instead, we will need to employ a cheap trick.    
    This trick runs as follows. First, we note that (assuming without loss of generality that $q$ terminated at $\gamma_2$)  the witness $w_3$ intersects the $\sigma$-neighbourhood of $q_1$, by Lemma~\ref{lem:witness-intersects-q}.
    Similarly, we note (assuming further without loss of generality that $q'$ also terminated at $\gamma_2$) that $w_1$ must intersect the $\sigma$-neighbourhood of $q'$.
    Now, let $r' > 10r$ be large enough that $N_{r'}(v_0)$ contains both $q$ and $q'$. Let $B'$ denote the ball of radius $100\sigma + r'$ around $v_0$. We play the same game again with $B'$ instead of $B$. Either we find our third missing path and we are done immediately, or we end up in the situation depicted in Figure~\ref{fig:crosshair-construction-badcase}. In this case, we are also done, as within this figure our desired construction is also hiding. This is highlighted in Figure~\ref{fig:crosshair-cheap-fix}. It is easy to verify that this is a tripod configuration.
\end{proof}

    \begin{figure}
\begin{minipage}{.5\textwidth}
  \centering

    \input{figures/crosshair-bad-case}

  \captionof{figure}{The bad case.}
  \label{fig:crosshair-construction-badcase}
\end{minipage}%
\begin{minipage}{.5\textwidth}
  \centering

\input{figures/crosshair-cheap-fix}

  \captionof{figure}{A cheap trick.}
  \label{fig:crosshair-cheap-fix}
\end{minipage}
\end{figure}

\section{Witness separation}\label{sec:witness-sep}

We now prove a lemma related to the witnesses of a tripod configuration, which we have dubbed the \emph{witness separation lemma}. Morally, this lemma says that any two of the three witnesses held by a tripod configuration are kept apart from each other by appropriate quasi-geodesics.

\begin{lemma}[The witness separation lemma]\label{lem:witness-separation}
    Let $\crosshair = (\gamma_i, w_i, q_i)$ be a tripod configuration. Then $(w_i,w_{i+1})$ is a pair of witnesses for $\rho_{i} := \gamma_{i+1}^{-1} \cdot \gamma_{i+2}$. 
\end{lemma}

\begin{proof}
    We assume the lemma fails and aim to find a contradiction to Lemma~\ref{lem:witness-intersects-q}.
    Suppose, without loss of generality, that $(w_1,w_2)$ is not a pair of witnesses for $\rho_1 = \gamma_2^{-1}\cdot \gamma_3$. Then there exists a path $p$, avoiding the $r$-neighbourhood of $v_0$ and the $\sigma$-neighbourhood of $\rho_1$, between the infinite components of $w_1 \setminus N_r(v_0)$ and $w_2 \setminus N_r(v_0)$.
    Since $(w_1,\gamma_1)$ is a pair of witnesses for $\rho_1$, the path $p$ cannot intersect the $\sigma$-neighbourhood of $\gamma_1$ either. 
    We also assume without loss of generality that $p$ does not pass $\sigma$-close to $w_3$. If it did, we shorten $p$ and consider the pair $(w_1,w_3)$ instead.

    Consider the path  $q$ from $\gamma_1$ to $\gamma_2$ depicted  in Figure~\ref{fig:witness-separation}. This path is well-defined as by Lemma~\ref{lem:witness-intersects-q}, we see that each $w_i$ intersects the $\sigma$-neighbourhood of $q_i$, while at distance at least $10\sigma$ from $\rho_i$ by \ref{ce4}. In particular, $q$ is far away from both $w_3$ and $\gamma_3$. Relating this to the quasi-geodesic $\rho_3$ and its pair of witnesses $(\gamma_3, w_3)$, this path satisfies the hypotheses of Lemma~\ref{lem:witness-intersects-q}. However, it remains far away from both witnesses, a contradiction. 
    \begin{figure}[ht]
        \centering
        \input{figures/witness-separation}
        \caption{Proof of the witness separation lemma~(\ref{lem:witness-separation}).}
        \label{fig:witness-separation}
    \end{figure}
\end{proof}

We now apply the witness separation lemma~(\ref{lem:witness-separation}) and prove the following technical result. The purpose of the upcoming lemma is to aid in the application of Lemma~\ref{lem:lies-in-bounded-component}. 

\begin{lemma}\label{lem:good-subpath-weak}
    Let $\crosshair = (v_0,\gamma_i, w_i, q_i)$ be an $(r,R,\lambda, c)$-tripod configuration in $X$. Let $p \subset X$ be a path such that:
    \begin{enumerate}
        \item $p$ avoids $N_R(v_0)$,

        \item $p$ starts at a point on $\gamma_1$, finishes on $\gamma_2$, and intersects $\gamma_3$ on the way. 
    \end{enumerate}
    Then, there exists a permutation $\pi$ of the set $\{1,2,3\}$ and a subpath $p' \subset p$ such that:
    \begin{enumerate}
        \item $p'$ starts in $N_\sigma(\gamma_{\pi(1)})$, finishes on $N_\sigma(\gamma_{\pi(2)})$, and intersects $N_\sigma(\gamma_{\pi(3)})$.

        \item $p'$ avoids the $2\sigma$-neighbourhood of $w_{\pi(3)}$. 
    \end{enumerate}
\end{lemma}

\begin{proof}
    We form a finite graph $\Delta$ with six vertices. The vertices correspond to the rays $\gamma_i$, $w_i$, for $i = 1,2,3$. We colour the $\gamma_i$-vertices in red, and the $w_i$-vertices in green. Connect two vertices together with an edge if there exists a path between the corresponding rays in $X$ which avoids the $2\sigma$-neighbourhood of the other 4 rays.

    We now describe the structure of the graph $\Delta$. First, note that $w_i$ cannot share an edge with $\gamma_i$, as $(w_i,\gamma_i)$ is a pair of witnesses for the quasi-geodesic $\rho_i = \gamma_{i+1}^{-1} \cdot \gamma_{i+2}$. 
    Moreover, by the witness separation lemma~(\ref{lem:witness-separation}), any path from $w_i$ to $w_j$ for $i \neq j$ must intersect the $\sigma$-neighbourhood of either $\gamma_{i+1}$ or $\gamma_{i+2}$. In particular, in $\Delta$ the green vertex labelled by $w_i$ can only share an edge with the two red vertices labelled by $\gamma_{i+1}$ and $\gamma_{i+2}$.
    Moreover, by Lemma~\ref{lem:witness-intersects-q}, we cannot have that two red vertices are adjacent, say $\gamma_1$ and $\gamma_2$, as such a path would need to intersect either $w_3$ or $\gamma_3$. 
    Let us assume without loss of generality that each green vertex labelled by $w_i$ is connected to both of the two red vertices labelled by $\gamma_{i+1}$ and $\gamma_{i+2}$. This is the most general case possible, and the possibility that some edges are missing does not affect the remainder of this argument. The graph $\Delta$, under these assumptions, is a cycle  of length 6, alternating red and green vertices. This is depicted in Figure~\ref{fig:graph-delta}. 

    \begin{figure}[ht]
        \centering
        \input{figures/graph-delta}
        \caption{The graph $\Delta$ constructed in the proof of Lemma~\ref{lem:good-subpath-weak}.}
        \label{fig:graph-delta}
    \end{figure}

    Given the path $p$ as in the statement of the lemma, note that $p$ can only be contained in the $2\sigma$-neighbourhood of at most one of the six rays at a time, as they are pairwise $10\sigma$-separated outside of $N_R(v_0)$. 
    Then, form \emph{induced $\Delta$-path} of $p$, denoted $\overline p$, in the graph $\Delta$ as follows:
    \begin{enumerate}
        \item The path $\overline p$ starts at the red vertex labelled by $\gamma_1$.

        \item As we travel the length of $p$, we keep track of the last of the $\{\gamma_i, w_i\}$ which we intersected the $2\sigma$-neighbourhood of. 
            Whenever this updates and we intersect the $2\sigma$-neighbourhood of another ray, we append the vertex labelled by this ray on to the end of $\overline p$.
    \end{enumerate}
    By the definition of $\Delta$, this forms a well-defined path through $\Delta$ which starts at the vertex labelled by $\gamma_1$, finishes at the vertex labelled by $\gamma_2$, and contains the vertex labelled by $\gamma_3$. 
    The proof of the lemma now reduces to the following claim. In this claim, we identify vertices of $\Delta$ with their labels to ease notation.

    \begin{claim}\label{claim:subpath-delta-weak}
        Let $q$ be a path in $\Delta$ which starts at $\gamma_1$, ends at $\gamma_2$, and contains $\gamma_3$. Then there exists a permutation $\pi$ of $\{1,2,3\}$ and a subpath $q'$, such that $q'$ starts at $\gamma_{\pi(1)}$, ends at $\gamma_{\pi(2)}$, and  does not contain $w_{\pi(3)}$. 
    \end{claim}

    \begin{proof}
        It is sufficient to prove that there exists a subpath of $q$ which contains every red vertex but only two green vertices. 
        Indeed, suppose that, say, $q$ does not contain $w_1$. Since $q$ contains both $\gamma_2$ and $\gamma_3$, we can consider a subpath which starts and ends at these vertices. Such a subpath must contain $\gamma_1$, and we are done.

        Suppose first that $q$ crosses every edge in $\Delta$. Then, there exists a prefix of $q$ which contains every red vertex but does not cross every edge. 
        Thus, we may assume without loss of generality that $q$ does not cross every edge. 
        In this case, $q$ is contained in a subgraph of $\Delta$ which is a path graph, where one end vertex is red and the other is green. Assume without loss of generality that the red vertex on the end is $\gamma_1$, and the green vertex on the other end is $w_2$. If $q$ does not contain $w_2$ we are done. Otherwise, we consider a subpath of $q$ of minimal length from $\gamma_1$ to $w_2$. Note that $w_2$ appears exactly once in this subpath, as the terminal vertex. But then, to reach $w_2$ we must have that this subpath hits every other red vertex. Thus, $q$ contains a subpath $q'$ which contains all red vertices but does not contain the green vertex $w_2$. 
    \end{proof}

    The lemma now follows immediately by applying Claim~\ref{claim:subpath-delta-weak} to the induced $\Delta$-path $\overline p$ of $p$, and lifting the given subpath of $\overline p$ to a subpath $p'$ of $p$ with the desired properties. 
\end{proof}

\section{Complementary depth}\label{sec:complementary depth}

We introduce some more terminology. 

\begin{definition}[Depth]
    Let $A \subset X$ be a connected subgraph. We denote by $\partial A$ the set of vertices at distance exactly 1 from $A$. The \emph{depth} of $A$, denoted $\depth(A)$, is defined as 
    $$
    \depth(A) := \inf \{r \geq 0 : A \subset N_r(\partial A)\}.
    $$
\end{definition}

\begin{definition}[Complementary depth]
    Given $\delta \geq 0$, let $S$ be a finite subgraph of $X$. Define the \emph{$\delta$-complementary depth}  of $S$, denoted $\operatorname{CD}_\delta(S)$, as 
    $$
    \operatorname{CD}_\delta(S) := \max_A \ \depth(A),
    $$
    where $A$ ranges over all finite connected components of $X \setminus N_\delta(S)$. 
    
    Let $\Lambda$ be a collection of finite subgraphs of $X$.
    We say that $\Lambda$ has \emph{uniformly bounded complementary depth} if for all $\delta \geq 0$ there exists $C \geq 0$ so that for all $S \in \Lambda$ we have $\operatorname{CD}_\delta(S) < C$. 
\end{definition}

We will show that, in a vertex-transitive graph with \property{}, not of quadratic growth, certain nice collections of subgraphs must have uniformly bounded complementary depth.
To prove this, we use the Coulhon--Saloff-Coste isoperimetric inequality, originating in work of Varopoulos. More precisely, we use Saloff-Coste's extension of this inequality to connected, locally finite, transitive graphs \cite{CoulhonSaloffCoste1993,saloff1995isoperimetric}. 

We need to set up some notation. 
Given a connected, locally finite, vertex-transitive graph $X$, let $\growth_X(n)$ denote the size of the closed $n$-ball around any given vertex. In other words, $\growth_X : \N \to \N$ is the \textit{growth function} of $X$. 
We say that $X$ has \emph{quadratic growth} if $\growth_X(n)$ grows quadratically.
Given a subset $A \subset V(X)$, we denote by $\partial A$ the set of vertices which lie at a distance of exactly 1 from $A$. 

\begin{theorem}[{\cite[Thm.~2.1]{saloff1995isoperimetric}}]
	Every infinite, connected, locally finite, vertex-transitive graph $X$ satisfies the following isoperimetric inequality.
	For all finite  $A \subset V(X)$, we have that 
	$$
	\frac{|A|}{\phi(2|A|)} \leq 4|\partial A|,
	$$
	where $\phi(\lambda) = \inf\{n \in \N : \growth_X(n) > \lambda\}$.\footnote{This function $\phi$ is sometimes referred to as the \emph{inverse growth function} in the literature.} 
\end{theorem}

This inequality has the consequence that if there exists $C >0$ and $D \in \N$ such that $\growth(n) \geq Cn^D$ for all $n \geq 1$, then there exists $C'> 0$ such that 
$$
|A|^{\frac{D}{D-1}} \leq C' |\partial A|.
$$
for all finite  $A \subset V(X)$. 
If we consider the contrapositive of this statement, it says that if there exists `large' subsets with suitably `small' boundaries, then this, in turn, presents us with a bound on the growth of $X$. More precisely, we have the following statement.

\begin{corollary}\label{cor:varopoulos}
    Suppose that there exists $K > 0$ and a sequence of finite subsets $A_n \subset X$ such that
    \begin{enumerate}
        \item $|\partial A_n| \to \infty$ as $n \to \infty$, and
        \item $ \depth(A_n) > \frac 1 K|\partial A_n| - K$.
    \end{enumerate}
    Then $X$ has quadratic growth. 
\end{corollary}

The next lemma provides an application of Corollary~\ref{cor:varopoulos}, specific to our setting. 

\begin{lemma}\label{lem:paths-on-crossexaminer-give-quadgrowth}
    Let $\crosshair = (v_0, \gamma_i,w_i, q_i)$ be an $(r,R,\lambda,c)$-tripod configuration. Suppose that there exists:
    \begin{enumerate}
        \item a strictly increasing sequence of integers $(d_n)$,

        \item a constant $K > 0$, and

        \item a sequence of finite connected subgraphs $Z_n$ of $X$,
    \end{enumerate}
    such that the following hold:
    \begin{enumerate}
        \item each $Z_n$ intersects all three of the $\gamma_i$,

        \item $|Z_n| < Kd_n + K$ for all $n > 0$,

        \item $\dist_X(v_0,Z_n) > \frac 1 Kd_n-K$ for all $n > 0$.
    \end{enumerate}
    Then $X$ has quadratic growth.
\end{lemma}

\begin{proof}
    Fix $n > 0$. Assume without loss of generality that $\dist_X(v_0,Z_n) > 5R + 2\sigma$ for all $n > 0$. Since $Z_n$ is connected, let $p_n \subset Z_n$ be a (not necessarily simple) path inside of $Z_n$ which starts at $\gamma_1$, ends at $\gamma_2$, and intersects $\gamma_3$.
    By Lemma~\ref{lem:good-subpath-weak}, there exists a permutation $\pi$ of the set $\{1,2,3\}$ and a subpath $p_n' \subset p_n$ such that
    \begin{enumerate}
        \item $p_n'$ starts in $N_\sigma(\gamma_{\pi(1)})$, finishes on $N_\sigma(\gamma_{\pi(2)})$, and intersects $N_\sigma(\gamma_{\pi(3)})$.

        \item $p_n'$ avoids the $2\sigma$-neighbourhood of $w_{\pi(3)}$. 
    \end{enumerate}
    Without loss of generality, assume that $\pi$ is the identity. 

    Extend $p_n'$ at both ends with a shortest path of length at most $\sigma$ to $\gamma_1$ and $\gamma_3$. Call this new extended path $p_n''$. Let $x \in \gamma_1$, $y \in \gamma_3$ be the new endpoints, and let $s_n$ be the subsegment of $\rho_3 = \gamma_1^{-1} \cdot \gamma_2$ from $y$ to $x$. Let $\ell_n = p_n'' \cdot s_n$. Since every point on $p_n''$ is at distance at least $\sigma$ from $w_3$, we have by Lemma~\ref{lem:lies-in-bounded-component} that the segment $\gamma_3'$ of $\gamma_3 \setminus N_\sigma(\ell_n)$ which has vertices adjacent to both $N_\sigma(s_n)$ and $N_\sigma(p_n'')$ lies in a bounded component of $X \setminus N_\sigma(\ell_n)$. 

    Write $D_n = \dist_X(v_0,p_n'')$, so that $D_n > \frac 1 Kd_n-K - \sigma$. 
    Let $t_n = D_n/2$, and let $z_n = \gamma_3(t_n)$. It is immediate that $z_n$ lies on $\gamma_3'$ for large enough $n > 0$. We also have the following claim.

    \begin{claim}
        $\dist_X(z_n,\ell_n) \geq \frac {1}{2K\lambda} d_n - K - c - 2\sigma$ for all $n > 0$.
    \end{claim}

    \begin{proof}
        Since paths are parametrised at unit speed, we see that 
        $$
        \dist_X(z_n, p_n'') > \frac{D_n}2 - \sigma > \frac 1 {2K}d_n-K - 2\sigma.
        $$
        Secondly, since $\gamma_3^{-1} \cdot \gamma_i$ is a $(\lambda,c)$-quasi-geodesic for $i = 1,2$, we have that
        $$
        \dist_X(z_n,s_n) > \frac 1 \lambda t_n - c \geq \frac {1}{2K\lambda} d_n - K - c.
        $$
        The claim follows, since $\ell_n = s_n \cup p_n''$ as a set.
    \end{proof}

    Let $A_n$ denote the bounded component of $X \setminus N_\sigma(\ell_n)$ which contains $z_n$. 

    \begin{claim}
        $|\partial A_n| < K' d_n + K'$, for some explicit constant $K'$ not depending on $n$.  
    \end{claim}

    \begin{proof}
        We can easily bound the size of $|\partial A_n|$ via
        \begin{align*}
            |\partial A_n| &\leq \growth_X(\sigma) (|\ell_n|) \\
            &\leq \growth_X(\sigma) (|p_n''| + |s_n|) \\
            &\leq \growth_X(\sigma) (|Z_n| + 2\sigma + \lambda (|Z_n| + 2\sigma) + c) \\
            &\leq \growth_X(\sigma) K (1+\lambda) d_n + \growth_X(\sigma)(c + (1+\lambda)(K+2\sigma)).
        \end{align*}
        This implies our claim.
    \end{proof}

    Applying Corollary~\ref{cor:varopoulos} to the bounded subsets $A_n$, we conclude that the graph $X$ has quadratic growth.
\end{proof}

\section{Quasi-circles}\label{sec:circles}

The following objects are of interest.

\begin{definition}[Quasi-circles]
    Let $X$ be a connected graph, $\lambda \geq 1$, $c \geq 0$. Then a closed loop $C \subset X$ is called a \emph{$(\lambda, c)$-quasi-circle} if, for any subpath $s \subset C$ of length at most half the total length of $C$, we have that $s$ is a $(\lambda, c)$-quasi-geodesic. 
\end{definition}

\begin{lemma}\label{lem:quasi-circles-coarsely-peripheral}
    Suppose $X$ does not have quadratic growth. Let $\lambda \geq 1$, $c \geq 0$.  Then, the collection of all $(\lambda, c)$-quasi-circles has uniformly bounded complementary depth. 
\end{lemma}

\begin{proof}
    Suppose that there exists $\lambda \geq 1$, $c \geq 0$ such that $(\lambda,c)$-quasi-circles do not have uniformly bounded complementary depth.  Then there exists $\delta > 0$, a sequence of $(\lambda, c)$-quasi-circles, and connected components $A_n$ of $X \setminus N_\delta(C_n)$ such that $\depth(A_n)$ goes to infinity. Let us suppose that this is indeed the case. 
    We will use Lemma~\ref{lem:paths-on-crossexaminer-give-quadgrowth} to deduce that $X$ has quadratic growth.
    
    For every $n > 0$, let $m_n \in A_n$ realise the depth of $A_n$. Write 
    $$
    d_n := \dist_X(m_n, C_n) = \depth(A_n)   + \delta.
    $$
    Fix an $(r,R,\lambda,c)$-tripod configuration $\crosshair = (v_0,\gamma_i, w_i, q_i)$, which exists by Proposition~\ref{prop:crosshair-exists}, and let $n > 0$ be large. We will assume without loss of generality that these are the same $\lambda$ and $c$ as above, as increasing the quasi-geodesic constants here does not affect the argument. Using the transitive group action, we may assume that $v_0 = m_n$ for all $n > 0$.

    For each $i \in \Z/3\Z$, let $t_i' \in [0,\infty)$ be minimal such that $\dist_X(\gamma_i(t_i'),C_n) \leq \delta$. This certainly exists, since $A_n$ is bounded and $\gamma_i$ is a quasi-geodesic ray. Let
    $$
    t_i = \min \{t_i',2\lambda(d_n+c)\}. 
    $$
    Write $\gamma_i' = \gamma_i |_{[0,t_i]}$. 
    Let $p_i$ be a shortest path of length at most $d_n$ from $\gamma_i(t_i)$ to $C_n$. Let $x_n$ be the endpoint of $p_i$ which lies on $C_n$. 

    Since $C_n$ is a $(\lambda,c)$-quasi-circle, we have that for all $i \in \Z/3\Z$ there exists a $(\lambda,c)$-quasi-geodesic $\alpha_i \subset C_n$ which connects $x_i$ to $x_{i+1}$. Let
    $$
    Z_n = p_1 \cup p_2 \cup p_3 \cup \alpha_1 \cup \alpha_2 \cup \alpha_3.
    $$
    Our intention is to apply Lemma~\ref{lem:paths-on-crossexaminer-give-quadgrowth} to these $Z_n$. We must check that the hypotheses are satisfied. 

    \begin{claim}
        There exists $K_1 > 0$, not depending on $n$, such that $|Z_n| < K_1 d_n + K_1$. 
    \end{claim}

    \begin{proof}
        Trivially, we have that $\length(p_i) \leq d_n$. By our choice of $t_i$, it is immediate that for $i \neq j$ we have 
        $$
        \dist_X(x_i,x_j) \leq 2(2\lambda (d_n + c) + d_n).
        $$
        In particular, we may bound
        $$
        \length(\alpha_i) \leq 2\lambda(2\lambda (d_n + c) + d_n) + \lambda c \leq 6\lambda^2 (d_n +c).
        $$
        Combining these bounds, we see that 
        $$
        |Z_n| \leq 18 \lambda^2 (d_n + c) + 3d_n = (18\lambda^2 + 3)d_n + 18\lambda^2 c.
        $$
        We may therefore take $K_1 = 21\lambda^2(1+c)$.
    \end{proof}

    \begin{claim}
        There exists $K_2 > 0$, not depending on $n$, such that $\dist_X(v_0, Z_n) > \tfrac 1 {K_2} d_n - K_2$. 
    \end{claim}

    \begin{proof}
        We trivially have that $\dist_X(v_0,\alpha_i) \geq d_n$. Consider now the $p_i$. If $t_i = t_i'$ then $\length(p_i) \leq \delta$, and so $\dist_X(v_0,p_i) \geq d_n - \delta$. Suppose instead that $t_i = 2\lambda(d_n + c)$. Then, since $\gamma_i$ is a $(\lambda,c)$-quasi-geodesic, we have that 
        $$
        \dist_X(v_0,p_i) \geq \frac 1 \lambda 2\lambda(d_n + c) - c - d_n \geq d_n.
        $$
        We may therefore take $K_2 = \max\{1,\delta\}$. 
    \end{proof}

    Each $Z_n$ trivially intersects all three of the $\gamma_i$. It follows from Lemma~\ref{lem:paths-on-crossexaminer-give-quadgrowth} that $X$ has quadratic growth. 
\end{proof}

Quasi-circles are not quite general enough for the final proof. We will require slightly more general figures called \emph{truncated quasi-circles}. 

\begin{definition}[Truncated quasi-circles]
    Let $\lambda \geq 1$, $c \geq 0$. Let $C$ be a $(\lambda, c)$-quasi-circle with a decomposition into two paths $C := F \cdot F'$. Let $f$ be a $(\lambda, c)$-quasi-geodesic with the same endpoints as $F'$. Then the figure $T := F \cdot f$ is called a \emph{truncated $(\lambda, c)$-quasi-circle}.  The subpaths $F$ and $f$ are called the \emph{major} and \emph{minor faces} of $T$.
    We sometimes write $T$ as an ordered pair $T = (F, f)$ to indicate which `side' has been truncated. 
\end{definition}

We now prove that truncated quasi-circles also have uniformly bounded complementary depth.

\begin{lemma}\label{lem:truncated-quasicircles-peripheral}
    Let $X$ be a connected, locally finite, vertex-transitive graph satisfying \property{} which does not have quadratic growth. Let $\lambda \geq 1$, $c \geq 0$.  Then the collection of all truncated $(\lambda, c)$-quasi-circles has uniformly bounded complementary depth.
\end{lemma}

\begin{proof}
    Suppose that there exists $\lambda \geq 1$, $c \geq 0$ such that truncated $(\lambda,c)$-quasi-circles do not have uniformly bounded complementary depth.  Then there exists $\delta > 0$, a sequence $T_n = (F_n, f_n)$ of truncated $(\lambda, c)$-quasi-circles, and connected components $A_n$ of $X \setminus N_\delta(T_n)$ such that $\depth(A_n)$ goes to infinity. 

    We proceed in much the same way as in the proof of Lemma~\ref{lem:quasi-circles-coarsely-peripheral}, though some extra casework will be required. The rough plan of attack is to proceed exactly as we did in the previous lemma, with a bit more care taken with how we choose our shortest paths from the tripod configuration to the truncated quasi-circle. In some cases, the same argument as the previous lemma will suffice, but will break down in others. If this happens, the extra control we established will allow us to construct new, smaller quasi-circles in controlled locations, within which we will apply Lemma~\ref{lem:quasi-circles-coarsely-peripheral} to conclude instead. The way we construct these smaller quasi-circles closely follows the proof of \cite[Prop.~5.1]{hume2020poorly}, though we do make some adjustments \emph{vis-\`a-vis} constants, since we are dealing with quasi-geodesics rather than geodesics.
    
    For every $n > 0$, let $m_n \in A_n$ realise the depth of $A_n$. That is, 
    $$
    d_n := \dist_X(m_n, T_n) = \depth(A_n) + \delta.
    $$
    Fix some large $n > 0$, and also fix a $(r,R,\lambda,c)$-tripod configuration $\crosshair = (v_0,\gamma_i, w_i,q_i)$, which exists by Proposition~\ref{prop:crosshair-exists}.  Using the vertex-transitive action, translate each $T_n$ so that $v_0 = m_n$. We assume without loss of generality and inconsequentially the error constants $\lambda$, $c$ for both the tripod configuration and the truncated quasi-circles are the same.

    Assume without loss of generality that $d_n > 2\lambda \delta$ for all $n > 0$. Consider the tripod configuration $\crosshair$. At this point we should remind the reader of our convention that all paths are parametrised at unit speed. For each $i \in \Z/3\Z$,  fix $t_i > 0$ so that 
    \begin{equation}\label{eq:t-i property}
        \frac{t_i}{\dist_X(\gamma_i(t_i), T_n)} \geq 2\lambda, 
    \end{equation}
    and $t_i$ is minimal among all such values with this property. Such $t_i$ certainly exist by our assumption that $d_n > 2\lambda \delta$, since the $\gamma_i$ must pass within $\delta$ of $T_n$ eventually.
    
    Let $p_i$ be a shortest path from $\gamma_i(t_i)$ to $T_n$. 
    Let $x_i$ denote the endpoint of $p_i$ which lies on $T_n$. 
    Write $\gamma_i' = \gamma_i|_{[0,t_i]}$. 
    We now prove a series of quantitative claims.

    \begin{claim}\label{claim:bound-on-ti}
        We have that 
        $
        \frac {2\lambda}{2\lambda+1} d_n \leq t_i \leq 2\lambda d_n
        $
        for all $i \in \Z/3\Z$. 
    \end{claim}

    \begin{proof}
        Since $t_i$ was chosen to be the minimal such argument satisfying (\ref{eq:t-i property}), we have that
        $
        t_i-1 < 2\lambda \dist_X(\gamma_i(t_i-1), T_n) \leq 2\lambda d_n
        $, 
        and so $t_i \leq 2\lambda d_n$. 
        Suppose now that $t_i < \frac {2\lambda}{2\lambda+1} d_n$. By the reverse triangle inequality, it would follow that 
        \begin{align*}
            2\lambda\dist_X(\gamma_i(t_i), T_n) &\geq 2\lambda (\dist_X(v_0, T_n) - \dist_X(v_0,\gamma_i(t_i)) )\\
            &> 2\lambda (d_n - \frac {2\lambda}{2\lambda+1} d_n) \\
            &> t_i.
        \end{align*}
        This contradicts (\ref{eq:t-i property}). 
    \end{proof}

    \begin{claim}\label{claim:bound-on-pi-length}
        The geodesics $p_i$ satisfy $2\lambda\length(p_i) \leq  t_i$.
    \end{claim}

    \begin{proof}
        This follows immediately from (\ref{eq:t-i property}).
    \end{proof}

    \begin{claim}\label{claim:bound-on-pi-dist}
        Given distinct $i , j \in \Z/3\Z$, we have that $\dist_X(p_i,p_j) \geq \frac 2{2\lambda+1} d_n - c$. 
    \end{claim}

    \begin{proof}
        By the reverse triangle inequality and Claims~\ref{claim:bound-on-ti} and \ref{claim:bound-on-pi-length}, we see that 
        \begin{align*}
            \dist_X(p_i,p_j) &\geq \frac 1 \lambda (t_i + t_j) - c - \length(p_i) - \length(p_j) \\
            &\geq \frac 1 \lambda (t_i + t_j) - c  - \frac 1 {2\lambda }(t_i + t_j)\\
            & \geq \frac{t_i + t_j}{2\lambda} - c \\
            & \geq \frac{2d_n}{2\lambda + 1} - c,
        \end{align*}
        as claimed.
    \end{proof}

    \begin{claim}\label{claim:bound-on-dist-from-gamma-to-Tn}
        For all $i \in \Z/3\Z$, we have that 
        $\dist_X(\gamma_i', T_n) > \frac{d_n}{2\lambda+1}$.
    \end{claim}

    \begin{proof}
        Let $a \in [0,t_i]$. By the reverse triangle inequality, we have that 
        $$
        \dist_X(\gamma_i(a), T_n) \geq \dist_X(v_0,T_n) - a = d_n - a.
        $$
        Since $t_i$ is the minimal argument satisfying (\ref{eq:t-i property}), we have 
        $
        \dist_X(\gamma_i(a), T_n) > \frac{a}{2\lambda} 
        $. 
        If we combine these observations, we see that $$
        \dist_X(\gamma_i(a), T_n) \geq \max\{d_n-a, \frac a {2\lambda}\} \geq \frac {d_n}{2\lambda + 1}.
        $$
        Since $a$ was arbitrary, the claim follows. 
    \end{proof}
    
    We must now separately consider two distinct possibilities. 

    \medskip

    \textbf{Case 1.} Suppose first that, for infinitely many $n > 0$, there exists a $(\lambda,c)$-quasi-geodesic $\alpha$ contained in $T_n$ which connects, say, $x_1$ to $x_2$. Consider the quasi-geodesic quadrilateral 
    $$
    Q_n = \gamma_2'^{-1} \cdot \gamma'_1 \cdot p_1 \cdot \alpha \cdot p_2^{-1}.
    $$
    We remark that we call $Q_n$ a quadrilateral, since $\gamma_2'^{-1} \cdot \gamma'_1$ is a $(\lambda, c)$-quasi-geodesic. 
    To ease notation for the time being, we will simply write $Q = Q_n$, and re-introduce the index later. 
    Note that we easily have the bound
    \begin{align*}
        \length(Q) &\leq t_1 + t_2 + \frac{t_1}{2\lambda} + \frac{t_2}{2\lambda} + \lambda \left( t_1 + t_2 + \frac{t_1}{2\lambda} + \frac{t_2}{2\lambda}\right) + c \\
        &\leq 16\lambda^2d_n + c.
    \end{align*}
    Since $Q$ is a connected subgraph of $X$, let $\dist_Q$ denote the intrinsic path metric inside of $Q$. 
    We now prove the following claim.

    \begin{claim}\label{claim:quasi-circle}
        We have that $Q$ is a $(\lambda', c')$-quasi-circle, where $\lambda' = 48\lambda^3$ and $c' = 2c$. 
    \end{claim}

    \begin{proof}
        We must show that the inclusion $Q \into X$ is a quasi-isometric embedding with uniform constants not depending on $n$. 
        
        Let $x,y \in Q$. If $x$, $y$ lie in a common `side' of $Q$ then we are trivially done. Suppose then that $x \in p_1$, $y \in p_2$. Then by Claim~\ref{claim:bound-on-pi-dist} and the earlier bound on $\length(Q)$, we have that 
        $$
        \dist_X(x,y) \geq \frac{1}{8\lambda^2(2\lambda+1)} \dist_Q(x,y) - 2c \geq \frac 1 {24\lambda^3} \dist_Q(x,y) - 2c.
        $$
        Similarly, if $x \in \gamma_1' \cup \gamma_2'$ and $y \in \alpha \subset T_n$, then by Claim~\ref{claim:bound-on-dist-from-gamma-to-Tn} we may compute
        $$
        \dist_X(x,y) \geq \frac 1 {48\lambda^3} \dist_Q(x,y) - c.
        $$
        Now, suppose without loss of generality $y \in \gamma_1' \cup \gamma_2'$ and $x \in p_1$, the case where $x \in p_2$ is identical. 
        If $y \in \gamma_2'$, then we can easily compute 
        \begin{align*}
            \dist_X(x,y) &\geq \frac 1 \lambda t_1 - c - \length(p_1) \\
            &\geq \frac{d_n}{2\lambda+1} - c \\
            & \geq \frac{1}{48\lambda^3} \dist_Q(x,y) - 2c.
        \end{align*}
        Suppose instead that $y \in \gamma_1'$. To ease calculation, let us reverse and re-parametrise $\gamma_1'$ so that $\gamma_1'(0) = p_1(0)$. Then, write $x = p_1(a)$, $y = \gamma_1'(b)$. Let $D = \dist_X(x,y)$. 
        Suppose for the sake of a contradiction that $D < \frac 1 {8\lambda} (a+b) - c$.
        By our choice of $t_1$ earlier as the minimal argument satisfying (\ref{eq:t-i property}), we have that 
        $$
            t_1 - b < 2\lambda \dist_X(y, T_n) \leq 2\lambda (\dist_X(x,y) + \dist_X(x,T_n)) = 2\lambda(D + \length(p_1) - a).
        $$
        As $\length(p_1) \leq \frac {t_1}{2\lambda}$, we see that
        $t_1 - b < 2\lambda D - 2\lambda a + t_1$ , and so $2\lambda a < b + 2\lambda D$. 
        Applying our hypothesised upper bound on $D$, routine  calculation shows that 
        $
        a < \frac 5{7\lambda} b
        $. 
        We now compute 
        \begin{align*}
            D &\geq \dist_X(y,p_1(0)) - \dist_X(x,p_1(0)) \\
            &\geq \frac 1 \lambda b - a - c \\
            &\geq \frac 2 {7\lambda }b - c \\
            &= \frac {2b}{7\lambda a + 7\lambda b}(a+b) - c \\
            &= \frac {1}{6\lambda}(a+b) - c \\
            &> D.
        \end{align*}
        This is clearly a contradiction, and so $\dist_X(x,y) \geq \frac 1 {8\lambda} \dist_Q(x,y) -c$.
        
        Finally, let us assume that $x \in p_1$ and $y \in \alpha$. Again, to ease calculation let us now parametrise $p_1$ and $\alpha$ so that $\alpha(0) = p_1(0)$. Write $x = p_1(a)$, $y = \alpha(b)$, and $D = \dist_X(x,y)$. We split into two cases. Suppose first that $a \geq \frac b{2\lambda}$. Then, since $p_1$ is a shortest path to $T_n$ (and thus $\alpha)$, we easily see that 
        $$
        D  \geq a = \frac{\lambda a}{3\lambda a} + \frac{2\lambda a}{3\lambda a} \geq \frac{1}{3\lambda}(a+b).
        $$
        Suppose now that $a < \frac{b}{2\lambda}$. More routine calculation shows that
        \begin{align*}
            D &\geq \dist_X(y,p_1(0)) - \dist_X(p_1(0),x) \\
            & \geq \frac1 \lambda b - c- a \\
            & \geq \frac 1 {3\lambda}(a+b) -c.
        \end{align*}
        We have considered all possible cases, whence the claim.
    \end{proof}

    We now aim to deduce quadratic growth. We return to our tripod configuration $\crosshair$. Let us write $s_n = p_1 \cdot \alpha \cdot p_2^{-1}$. Note that each $s_n$ is a path which connects $\gamma_1$ and $\gamma_2$ (without loss of generality). Moreover, we have the following extra observations.

    \begin{claim}\label{claim:sn-properties}
        The path $s_n$ satisfies the following:
        \begin{enumerate}
            \item\label{itm:length-sn} $\length(s_n) \leq 16\lambda^2d_n + c$.

            \item\label{itm:dist-sn} $\dist_X(v_0,s_n) \geq \frac{2\lambda}{\lambda'(2\lambda+1)}d_n - c'$.
        \end{enumerate}
    \end{claim}

    \begin{proof}
        The first property follows trivially from the fact that $s_n$ is a subpath of $Q_n$. The second follows from Claims~\ref{claim:quasi-circle} and \ref{claim:bound-on-ti}.
    \end{proof}

    Suppose now that $s_n$ intersects the $2\sigma$-neighbourhood of $\gamma_3$ for infinitely many $n$. In this case, it follows that $X$ has quadratic growth by Claim~\ref{claim:sn-properties} and Lemma~\ref{lem:paths-on-crossexaminer-give-quadgrowth}. Otherwise, we may assume that $s_n$ is disjoint from $N_{2\sigma}(\gamma_3)$ for all sufficiently large $n > 0$. Now, it follows from Lemma~\ref{lem:lies-in-bounded-component} that an increasingly large initial segment of the witness $w_3$ lies in a bounded component of $X \setminus N_{2\sigma}(Q_n)$. Taking $n \to \infty$, we see that the $2\sigma$-complementary depth of the $Q_n$ becomes arbitrarily large, since $w_3$ diverges from $\gamma_1$ and $\gamma_2$ and $\dist_X(v_0, s_n) \to \infty$ as $n \to \infty$. However, by Lemma~\ref{lem:quasi-circles-coarsely-peripheral}, this implies again that $X$ must have quadratic growth. This concludes case 1.

    \medskip

    \textbf{Case 2.} Suppose now that, for infinitely many $n > 0$, there does not exist $(\lambda,c)$-quasi-geodesic $\alpha$ contained in $T_n$ which connects any two distinct $x_i$, $x_j$.
    This can only happen in one certain arrangement. In particular, this implies that, up to relabelling, we have that $x_1$ and $x_2$ lie on the major face $F_n$ of $T_n$, $x_3$ lies on the minor face $f_n$, and the segment of $F_n$ connecting $x_1$ to $x_2$ is not a $(\lambda,c)$-quasi-geodesic. Let $C_n$ denote the $(\lambda,c)$-quasi-circle of which $F_n$ is a subpath of. Then there must exist a subsegment $\beta_n$ of $C_n$ connecting $x_1$ to $x_2$ which \emph{is} a $(\lambda,c)$-quasi-geodesic. Recall that, to obtain $T_n$, we have replaced a subsegment of $\beta_n$ with $f_n$. Let $\alpha_n \subset T_n$ denote the subpath of $T_n$ which is obtained in this way. Now, while $\alpha_n$ itself is not in general going to be a quasi-geodesic, it does still have sufficiently nice properties which will allow us to apply Lemma~\ref{lem:paths-on-crossexaminer-give-quadgrowth}. 
    Consider the connected subgraph 
    $$
    Z_n = \alpha_n \cup p^{(n)}_1 \cup p^{(n)}_2 \cup p^{(n)}_3.
    $$
    We note the following.

    \begin{claim}\label{claim:Zn properties}
        There exists $K > 0$ such that for all $n > 0$, we have the following:
        \begin{enumerate}
            \item $Z_n$ intersects all three of $\gamma_1$, $\gamma_2$, and $\gamma_3$. 

            \item $|Z_n| \leq Kd_n + K$.

            \item $\dist_X(v_0,Z_n) \geq \frac 1 K d_n - K$. 
        \end{enumerate}
    \end{claim}

    \begin{proof}
        The first item is entirely trivial, and the third follows from earlier analysis. The second item follows from the fact that, while $\alpha_n$ is not itself a quasi-geodesic, it is obtained by taking a $(\lambda,c)$-quasi-geodesic $\beta_n$ and replacing a middle segment with another $(\lambda,c)$-quasi-geodesic, namely the minor face $f_n$. The length of $\beta_n$ is easily seen to be bounded above by an affine function of $d_n$ depending only on $\lambda$ and $c$, and so the same is clearly true for $\alpha_n$. We leave exact computation of the constant $K$ here as an exercise.  
    \end{proof}
    
    It now follows immediately from Claim~\ref{claim:Zn properties} and Lemma~\ref{lem:paths-on-crossexaminer-give-quadgrowth} that $X$ has quadratic growth, whence the lemma.
\end{proof}

\section{Proof of the main theorem}\label{sec:final}

In this section, we now proceed with the proof of our main result. 
We will mainly be concerned with the non-hyperbolic case, and for this we will need the following result.

\begin{proposition}[{\cite[Prop.~5.1]{hume2020poorly}}]\label{prop:quasi-circles-exist}
    If $X$ is not hyperbolic then $X$ contains arbitrarily long $(18,0)$-quasi-circles. 
\end{proposition}

The next lemma is the final ingredient in our proof.

\begin{lemma}\label{lem:bq-implies-quadgrowth}
     If $X$ is not hyperbolic then it has quadratic growth. 
\end{lemma}

\begin{proof}
    Assume that $X$ is not hyperbolic. Our strategy of proof is to find a sequence of uniform truncated quasi-circles which do not have uniformly bounded complementary depth. This will imply that $X$ has quadratic growth by Lemma~\ref{lem:truncated-quasicircles-peripheral}. 

    By Proposition~\ref{prop:quasi-circles-exist}, there exists a sequence $(C_n)$ of $(18,0)$-quasi-circles such that $\length(C_n)$ tends to infinity. Fix a base-vertex $u_n$ in every $C_n$. Let $p_n$ be a subpath of $C_n$ such that $u_n$ is the midpoint of $p_n$, and $\length(p_n) > \length(C_n) / 4$. 
    Using the vertex-transitivity of $X$, we may assume that 
    $
    u_1 = u_2 =  \ldots
    $, and so on. 
    Write $u = u_1$ to ease notation. 
    By the Arzel\`a--Ascoli theorem, the $p_n$ converge uniformly on compact sets to a bi-infinite $(18, 0)$-quasi-geodesic $\rho$. 
    Let 
    $$
    r = \max\{\dist_X(w_i(0), u) : i = 1,2\} + 10\sigma, 
    $$
    and let $B = N_r(u)$. Let $\rho_+, \rho_- \subset \rho$ denote the two infinite connected components of $\rho \setminus B$. 
    Let $(w_1, w_2)$ be a pair of witnesses for $\rho$. Assume without loss of generality that $w_1$ and $w_2$ lie outside of the $10\sigma$-neighbourhood of $\rho$.

    \begin{claim}\label{claim:lands-either-side}
        For every sufficiently large $n > 0$, there exists a subsegment $a_n$ of $C_n$ satisfying the following:
        \begin{enumerate}
            \item $a_n$ is disjoint from $N_\sigma(\rho)$.

            \item The two endpoints $x_n$, $y_n$ of $a_n$ lie $(\sigma+1)$-close to $\rho_+$ and $\rho_-$, respectively. 

            \item $\dist_X(u, a_n) \to \infty$ as $n \to \infty$. 
        \end{enumerate}
    \end{claim}

    \begin{proof}
        There are balls $B_n$ of increasing radii based at $u$ such that $C_n \cap B_n = \rho \cap B_n$.  Outside of this ball, $C_n$ clearly passes $\sigma$-close to both $\rho_+$ and $\rho_-$. Thus, by considering the path segment $C_n \setminus B_n$, this must contain such an $a_n$. 
    \end{proof}

    Let us assume without loss of generality that Claim~\ref{claim:lands-either-side} holds for all $n$. 
    Now, by Lemma~\ref{lem:witness-intersects-q} we have for all $n> 0$ that either $w_1$ or $w_2$ intersects $N_\sigma(a_n)$. By passing to a subsequence and possibly relabelling, let us assume without loss of generality $w_1$ intersects $N_{\sigma}(a_n)$ for every $n$. 
    Consider $n = 1$, and let
    $$
    R = \max \{\dist_X(z, u) : z \in a_1\} + 10\sigma. 
    $$
    Let $B' = N_R(u)$. We pass to yet another subsequence, and assume for all $n > 1$ that $C_n$ coincides with $\rho$ inside of $B'$. See Figure~\ref{fig:final-proof-set-up} for a cartoon of this set-up. 

    \begin{figure}
        \centering
        \input{figures/final-proof-setup}
        \caption{}
        \label{fig:final-proof-set-up}
    \end{figure}

    Extend each $a_n$ with two geodesics of length $\sigma+1$ so that its endpoints lie on $\rho$. Call the resulting path $a_n'$.
    For all $n > 1$, let $s_n$ denote the subsegment of $\rho$ connecting the endpoints of $a_n'$. Let $T_n = s_n \cdot a_n$, and let $A_n$ denote the connected component of $X \setminus N_{\sigma}(T_n)$ which contains $w_1(0)$. By Lemma~\ref{lem:lies-in-bounded-component}, we can see that $A_n$ is bounded. It is immediate that each $T_n$ is a truncated $(18, 5\sigma+5)$-quasi-circle.



    \begin{claim}
        We have that $\depth(A_n) \to \infty$ as $n \to \infty$. 
    \end{claim}

    \begin{proof}
        Consider $s_n$. It is clear that for all $t > 0$, there exists a natural number $N$ such that for all $n \geq N$, we have that the nearest point projection of $w_1(t)$ onto $T_n$ lies in $s_n$. Since $w_1$ diverges from $\rho$, this proves the claim. 
    \end{proof}

    By Lemma~\ref{lem:truncated-quasicircles-peripheral}, this implies that $X$ has quadratic growth, as we have shown that  truncated $(18,5\sigma+5)$-quasi-circles do not have uniformly bounded complementary depth.
\end{proof}

\main*

\begin{proof}
    Firstly, note that we may assume without loss of generality that $X$ is vertex-transitive, rather than just quasi-transitive, as \property{} is clearly a quasi-isometry invariant and every connected, locally finite, quasi-transitive graph is quasi-isometric to a locally finite, vertex-transitive graph, namely the Cayley--Abels graph of its isometry group \cite[\S2.E]{Cornulier2014}. 

    Either $X$ is hyperbolic, or it is not. First, let us suppose that $X$ is not hyperbolic. 
    By Lemma~\ref{lem:bq-implies-quadgrowth}, we have that $X$ has quadratic growth. We now apply Trofimov's theorem \cite[Thm.~1]{trofimov1985graphs}, which is an extension of Gromov's theorem for groups of polynomial growth \cite{gromov1981groups} to vertex-transitive graphs, and deduce that $X$ is quasi-isometric to a virtually nilpotent group $G$, necessarily of quadratic growth. By the Bass--Guivarc'h formula for the degrees of growth of nilpotent groups \cite{bass1972degree, guivarc1973croissance}, it follows that $G$ is virtually $\Z^2$ and so $X$ is quasi-isometric to the Euclidean plane.

    Now, let us assume that $X$ is hyperbolic. 
    We will make use of standard facts about the Gromov boundary. 
    By Proposition~\ref{prop:one-ended}, we have that $X$ is one-ended and so $\partial_\infty X$ is connected. 
    It is immediate from \property{} and the visibility property that every pair of distinct points $p\neq q$ in $\partial_\infty X$ is a cut pair. That is, $\partial_\infty X \setminus \{p,q\}$ is disconnected.
    Since $\partial_\infty X$ is compact, this implies that $\partial_\infty X$ is homeomorphic to the circle $\mathbf S^1$; see, for example, \cite[Thm.~IV.12.1]{newman1951elements}.  
    It now follows from known results that $X$ is quasi-isometric to the hyperbolic plane; see the proof of \cite[Thm.~1.1]{macmanus2024note} for details on how to deduce this in the case that $X$ is not a Cayley graph.
\end{proof}

\section{Closing remarks}

We close this paper with a question. In the introduction, we remarked that the convergence group theorem implies that hyperbolic groups with \property{} are virtually Fuchsian. One might protest that the consequences here are actually stronger. 

\begin{definition}
    Let $X$ be a geodesic metric space. We define the \emph{uniformly bisecting geodesics} property, denoted (\ref{eq:BG}), as follows:
        \begin{equation}
  \tag{$\rm{UBG}$}\label{eq:BG}
  \parbox{4in}{%
    \strut
    There exists $\sigma > 0$ such that for any connected, bi-infinite \textbf{geodesic} $\rho \subset X$, we have that $X \setminus N_\sigma(\rho)$ contains exactly two wide connected components. 
    \strut
  }
\end{equation}
\end{definition}

In other words, the definition of (\ref{eq:BG}) is obtained by simply taking Definition~\ref{def:ubq} and replacing the word `quasi-geodesic' with `geodesic'.
Note that (\ref{eq:BG}) is \emph{a priori} much weaker than \property{}. However, it follows from the convergence group theorem and standard properties of hyperbolic groups that an infinite hyperbolic group is virtually Fuchsian if and only if its Cayley graphs have (\ref{eq:BG}). This begs the natural question.

\begin{question}
    Let $X$ be an infinite, connected, locally finite, quasi-transitive graph satisfying (\ref{eq:BG}). Then, is $X$ quasi-isometric to either the Euclidean plane or the hyperbolic plane?
\end{question}

I expect the answer to be negative in general, but perhaps positive under the right additional hypotheses. For example, it is not obvious that (\ref{eq:BG}) should be invariant under quasi-isometry. Very little is known about the structure of geodesic rays in an arbitrary locally finite, quasi-transitive graph, and the outlook is relatively bleak. It could be that there is a counterexample to the above question which just contains `very few' bi-infinite geodesics, by some measure. This would be very interesting to see.

We could also ask a slightly weaker question. Perhaps instead of considering \emph{all} quasi-geodesics like we did in Definition~\ref{def:ubq}, we just consider $(\lambda,c)$-quasi-geodesics where $\lambda$ and $c$ are some fixed universal constants. It is plausible to imagine that if, say, all $(100,100)$-quasi-geodesics uniformly coarsely bisect a locally finite, quasi-transitive graph $X$, then $X$ is quasi-isometric to a plane. The existence of such universal constants does not follow from our proof of Theorem~\ref{thm:main-theorem}.

\bibliography{references}

\end{document}

%% file: commands.tex
\DeclareMathOperator{\length}{length}

\DeclareMathOperator{\dHaus}{Haus}

\DeclareMathOperator{\dist}{d}

\newcommand{\R}{\mathbf{R}}
\newcommand{\Z}{\mathbf{Z}}
\newcommand{\N}{\mathbf{N}}

\newcommand{\into}{\hookrightarrow}

\makeatletter
\newcommand{\myitem}[1]{%
\item[#1]\protected@edef\@currentlabel{#1}%
}
\makeatother

\newtheorem{theorem}{Theorem}[section]
\newtheorem*{theorem*}{Theorem}

\newtheorem{proposition}[theorem]{Proposition}
\newtheorem{lemma}[theorem]{Lemma}
\newtheorem{corollary}[theorem]{Corollary}
\newtheorem{claim}[theorem]{Claim}
\newtheorem{question}[theorem]{Question}

\newtheorem*{question*}{Question}

\theoremstyle{definition}

\newtheorem{definition}[theorem]{Definition}

\newtheorem{remark}[theorem]{Remark}
\newtheorem{convention}[theorem]{Convention}

\newtheorem*{definition*}{Definition}

\newcommand{\crosshair}{\mathcal{C}}

\DeclareMathOperator{\depth}{depth}

\newcommand{\property}{(\ref{eq:BQ})}

\newcommand{\growth}{\mathfrak{G}}

%% file: figures/witness-prot-first-app.tex
\tikzset{every picture/.style={line width=0.75pt}} 

\begin{tikzpicture}[x=0.75pt,y=0.75pt,yscale=-1,xscale=1]

\draw  [color={rgb, 255:red, 155; green, 155; blue, 155 }  ,draw opacity=1 ][dash pattern={on 1.5pt off 1.5pt}] (213.25,197) .. controls (213.25,187.47) and (220.97,179.75) .. (230.5,179.75) .. controls (240.03,179.75) and (247.75,187.47) .. (247.75,197) .. controls (247.75,206.53) and (240.03,214.25) .. (230.5,214.25) .. controls (220.97,214.25) and (213.25,206.53) .. (213.25,197) -- cycle ;
\draw [line width=1.5]    (156.72,196.32) .. controls (225.51,191.33) and (260.83,207.76) .. (299.71,197.93) ;
\draw [shift={(303.33,196.94)}, rotate = 163.62] [fill={rgb, 255:red, 0; green, 0; blue, 0 }  ][line width=0.08]  [draw opacity=0] (6.97,-3.35) -- (0,0) -- (6.97,3.35) -- cycle    ;
\draw [shift={(152.33,196.67)}, rotate = 355.19] [fill={rgb, 255:red, 0; green, 0; blue, 0 }  ][line width=0.08]  [draw opacity=0] (6.97,-3.35) -- (0,0) -- (6.97,3.35) -- cycle    ;
\draw [color={rgb, 255:red, 189; green, 16; blue, 224 }  ,draw opacity=1 ]   (224.25,172.38) .. controls (264.58,173.19) and (264.67,187.74) .. (272.67,200.99) ;
\draw [shift={(272.67,200.99)}, rotate = 58.89] [color={rgb, 255:red, 189; green, 16; blue, 224 }  ,draw opacity=1 ][fill={rgb, 255:red, 189; green, 16; blue, 224 }  ,fill opacity=1 ][line width=0.75]      (0, 0) circle [x radius= 3.35, y radius= 3.35]   ;
\draw    (236,205.88) .. controls (238.3,218.14) and (249.2,216.05) .. (231.31,234.76) ;
\draw [shift={(229.25,236.88)}, rotate = 314.66] [fill={rgb, 255:red, 0; green, 0; blue, 0 }  ][line width=0.08]  [draw opacity=0] (5.36,-2.57) -- (0,0) -- (5.36,2.57) -- cycle    ;
\draw [color={rgb, 255:red, 189; green, 16; blue, 224 }  ,draw opacity=1 ]   (210.25,173.63) .. controls (195.75,177.88) and (203,180.71) .. (196.67,195.86) ;
\draw [shift={(196.67,195.86)}, rotate = 112.69] [color={rgb, 255:red, 189; green, 16; blue, 224 }  ,draw opacity=1 ][fill={rgb, 255:red, 189; green, 16; blue, 224 }  ,fill opacity=1 ][line width=0.75]      (0, 0) circle [x radius= 3.35, y radius= 3.35]   ;
\draw    (227.25,189.88) .. controls (214.43,179.44) and (217.8,169.08) .. (205.76,159.53) ;
\draw [shift={(203.5,157.88)}, rotate = 34.14] [fill={rgb, 255:red, 0; green, 0; blue, 0 }  ][line width=0.08]  [draw opacity=0] (5.36,-2.57) -- (0,0) -- (5.36,2.57) -- cycle    ;
\draw [line width=1.5]    (373.72,198.57) .. controls (442.51,193.58) and (477.83,210.01) .. (516.71,200.18) ;
\draw [shift={(520.33,199.19)}, rotate = 163.62] [fill={rgb, 255:red, 0; green, 0; blue, 0 }  ][line width=0.08]  [draw opacity=0] (6.97,-3.35) -- (0,0) -- (6.97,3.35) -- cycle    ;
\draw [shift={(369.33,198.92)}, rotate = 355.19] [fill={rgb, 255:red, 0; green, 0; blue, 0 }  ][line width=0.08]  [draw opacity=0] (6.97,-3.35) -- (0,0) -- (6.97,3.35) -- cycle    ;
\draw    (441.25,174.63) .. controls (481.58,175.44) and (481.67,189.99) .. (489.67,203.24) ;
\draw    (453,208.13) .. controls (455.3,220.39) and (466.2,218.3) .. (448.31,237.01) ;
\draw [shift={(446.25,239.13)}, rotate = 314.66] [fill={rgb, 255:red, 0; green, 0; blue, 0 }  ][line width=0.08]  [draw opacity=0] (5.36,-2.57) -- (0,0) -- (5.36,2.57) -- cycle    ;
\draw    (427.25,175.88) .. controls (412.75,180.13) and (420,182.96) .. (413.67,198.11) ;
\draw    (444.25,192.13) .. controls (431.43,181.69) and (434.8,171.33) .. (422.76,161.78) ;
\draw [shift={(420.5,160.13)}, rotate = 34.14] [fill={rgb, 255:red, 0; green, 0; blue, 0 }  ][line width=0.08]  [draw opacity=0] (5.36,-2.57) -- (0,0) -- (5.36,2.57) -- cycle    ;
\draw    (444.25,192.13) -- (447.5,200.38) ;
\draw    (447.5,200.38) -- (453,208.13) ;
\draw [color={rgb, 255:red, 208; green, 2; blue, 27 }  ,draw opacity=1 ][line width=2.25]    (518.82,200.48) -- (507.01,202.53) -- (489.42,202.74) -- (484.25,192.88) -- (479.25,186.13) -- (475.25,182.63) -- (468.5,178.63) -- (462.5,176.63) -- (456.75,175.88) -- (449.75,175.13) -- (441.25,174.63) ;
\draw [shift={(523.75,199.63)}, rotate = 170.16] [fill={rgb, 255:red, 208; green, 2; blue, 27 }  ,fill opacity=1 ][line width=0.08]  [draw opacity=0] (11.43,-5.49) -- (0,0) -- (11.43,5.49) -- cycle    ;
\draw [color={rgb, 255:red, 74; green, 144; blue, 226 }  ,draw opacity=1 ][line width=1.5]    (422.31,161.2) -- (428.5,168.63) -- (435,180.13) -- (438,186.13) -- (444.25,192.13) -- (447.5,200.38) -- (453,208.13) -- (455.75,215.88) -- (458.75,220.38) -- (457.25,226.88) -- (452.5,232.13) -- (447.69,237.42) ;
\draw [shift={(445,240.38)}, rotate = 312.27] [fill={rgb, 255:red, 74; green, 144; blue, 226 }  ,fill opacity=1 ][line width=0.08]  [draw opacity=0] (9.29,-4.46) -- (0,0) -- (9.29,4.46) -- cycle    ;
\draw [shift={(419.75,158.13)}, rotate = 50.19] [fill={rgb, 255:red, 74; green, 144; blue, 226 }  ,fill opacity=1 ][line width=0.08]  [draw opacity=0] (9.29,-4.46) -- (0,0) -- (9.29,4.46) -- cycle    ;
\draw [color={rgb, 255:red, 208; green, 2; blue, 27 }  ,draw opacity=1 ][line width=2.25]    (427.25,175.88) -- (424.25,176.88) -- (420.75,178.38) -- (418.25,181.13) -- (416,186.88) -- (415,191.13) -- (413.42,197.61) -- (408.79,197.69) -- (394.77,197.3) -- (376,198.13) -- (372.42,198.76) ;
\draw [shift={(367.5,199.63)}, rotate = 349.99] [fill={rgb, 255:red, 208; green, 2; blue, 27 }  ,fill opacity=1 ][line width=0.08]  [draw opacity=0] (11.43,-5.49) -- (0,0) -- (11.43,5.49) -- cycle    ;
\draw    (227.25,189.88) .. controls (232.75,195.75) and (225,196.75) .. (236,205.88) ;
\draw [shift={(236,205.88)}, rotate = 39.68] [color={rgb, 255:red, 0; green, 0; blue, 0 }  ][fill={rgb, 255:red, 0; green, 0; blue, 0 }  ][line width=0.75]      (0, 0) circle [x radius= 2.01, y radius= 2.01]   ;
\draw [shift={(227.25,189.88)}, rotate = 46.89] [color={rgb, 255:red, 0; green, 0; blue, 0 }  ][fill={rgb, 255:red, 0; green, 0; blue, 0 }  ][line width=0.75]      (0, 0) circle [x radius= 2.01, y radius= 2.01]   ;

\draw (241.33,225.48) node [anchor=north west][inner sep=0.75pt]  [font=\footnotesize]  {$w_{2}$};
\draw (258.83,163.57) node [anchor=north west][inner sep=0.75pt]  [font=\footnotesize,color={rgb, 255:red, 189; green, 16; blue, 224 }  ,opacity=1 ]  {$q$};
\draw (163.25,177.4) node [anchor=north west][inner sep=0.75pt]  [font=\footnotesize]  {$\rho $};
\draw (316.33,172.4) node [anchor=north west][inner sep=0.75pt]    {$\Longrightarrow $};
\draw (381.5,175.4) node [anchor=north west][inner sep=0.75pt]  [font=\footnotesize,color={rgb, 255:red, 208; green, 2; blue, 27 }  ,opacity=1 ]  {$\rho '$};
\draw (215.08,147.73) node [anchor=north west][inner sep=0.75pt]  [font=\footnotesize]  {$w_{1}$};

\end{tikzpicture}

%% file: figures/An-bounded-case-1.tex
\tikzset{every picture/.style={line width=0.75pt}} 

\begin{tikzpicture}[x=0.75pt,y=0.75pt,yscale=-1,xscale=1]

\draw    (439.91,187.98) .. controls (442.45,172.15) and (428.53,163.2) .. (457.94,146.33) ;
\draw [line width=1.5]    (167.45,183.05) .. controls (238.27,180.23) and (271.73,196.16) .. (338.85,186.75) ;
\draw [shift={(341.94,186.3)}, rotate = 171.5] [fill={rgb, 255:red, 0; green, 0; blue, 0 }  ][line width=0.08]  [draw opacity=0] (6.97,-3.35) -- (0,0) -- (6.97,3.35) -- cycle    ;
\draw [shift={(163,183.25)}, rotate = 357.14] [fill={rgb, 255:red, 0; green, 0; blue, 0 }  ][line width=0.08]  [draw opacity=0] (6.97,-3.35) -- (0,0) -- (6.97,3.35) -- cycle    ;
\draw    (241.25,136.75) .. controls (248.34,132.12) and (254.95,138.32) .. (267.45,137.37) .. controls (279.95,136.42) and (301.98,136.27) .. (293.51,189.47) ;
\draw [shift={(293.51,189.47)}, rotate = 99.04] [color={rgb, 255:red, 0; green, 0; blue, 0 }  ][fill={rgb, 255:red, 0; green, 0; blue, 0 }  ][line width=0.75]      (0, 0) circle [x radius= 2.68, y radius= 2.68]   ;
\draw    (213.41,183.98) .. controls (215.95,168.15) and (202.03,159.2) .. (231.44,142.33) ;
\draw [shift={(213.41,183.98)}, rotate = 279.11] [color={rgb, 255:red, 0; green, 0; blue, 0 }  ][fill={rgb, 255:red, 0; green, 0; blue, 0 }  ][line width=0.75]      (0, 0) circle [x radius= 2.68, y radius= 2.68]   ;
\draw    (257.42,197.73) .. controls (259.39,208.69) and (265.04,208.69) .. (261.38,222.24) ;
\draw [shift={(260.56,225.01)}, rotate = 287.7] [fill={rgb, 255:red, 0; green, 0; blue, 0 }  ][line width=0.08]  [draw opacity=0] (5.36,-2.57) -- (0,0) -- (5.36,2.57) -- cycle    ;
\draw [color={rgb, 255:red, 155; green, 155; blue, 155 }  ,draw opacity=1 ] [dash pattern={on 2.25pt off 2.25pt}]  (254.92,179.5) .. controls (270,169.75) and (311.5,147.25) .. (299.5,144.5) .. controls (287.5,141.75) and (259.4,169.98) .. (259.25,136.25) ;
\draw [line width=1.5]    (259.25,136.25) .. controls (254.57,120.79) and (278.87,119.57) .. (254.51,95.69) ;
\draw [shift={(251.67,93)}, rotate = 42.42] [fill={rgb, 255:red, 0; green, 0; blue, 0 }  ][line width=0.08]  [draw opacity=0] (4.64,-2.23) -- (0,0) -- (4.64,2.23) -- cycle    ;
\draw [shift={(259.25,136.25)}, rotate = 253.15] [color={rgb, 255:red, 0; green, 0; blue, 0 }  ][fill={rgb, 255:red, 0; green, 0; blue, 0 }  ][line width=1.5]      (0, 0) circle [x radius= 1.74, y radius= 1.74]   ;
\draw    (254.92,179.5) -- (260.25,188.25) ;
\draw [shift={(260.25,188.25)}, rotate = 58.64] [color={rgb, 255:red, 0; green, 0; blue, 0 }  ][fill={rgb, 255:red, 0; green, 0; blue, 0 }  ][line width=0.75]      (0, 0) circle [x radius= 1.34, y radius= 1.34]   ;
\draw [shift={(254.92,179.5)}, rotate = 58.64] [color={rgb, 255:red, 0; green, 0; blue, 0 }  ][fill={rgb, 255:red, 0; green, 0; blue, 0 }  ][line width=0.75]      (0, 0) circle [x radius= 1.34, y radius= 1.34]   ;
\draw    (260.25,188.25) -- (257.42,197.73) ;
\draw [shift={(257.42,197.73)}, rotate = 106.62] [color={rgb, 255:red, 0; green, 0; blue, 0 }  ][fill={rgb, 255:red, 0; green, 0; blue, 0 }  ][line width=0.75]      (0, 0) circle [x radius= 1.34, y radius= 1.34]   ;
\draw [shift={(260.25,188.25)}, rotate = 106.62] [color={rgb, 255:red, 0; green, 0; blue, 0 }  ][fill={rgb, 255:red, 0; green, 0; blue, 0 }  ][line width=0.75]      (0, 0) circle [x radius= 1.34, y radius= 1.34]   ;
\draw [color={rgb, 255:red, 208; green, 2; blue, 27 }  ,draw opacity=1 ]   (264.25,109.5) .. controls (220.25,128.5) and (237,146.5) .. (254.92,179.5) ;
\draw [shift={(254.92,179.5)}, rotate = 61.5] [color={rgb, 255:red, 208; green, 2; blue, 27 }  ,draw opacity=1 ][fill={rgb, 255:red, 208; green, 2; blue, 27 }  ,fill opacity=1 ][line width=0.75]      (0, 0) circle [x radius= 1.34, y radius= 1.34]   ;
\draw [shift={(264.25,109.5)}, rotate = 156.64] [color={rgb, 255:red, 208; green, 2; blue, 27 }  ,draw opacity=1 ][fill={rgb, 255:red, 208; green, 2; blue, 27 }  ,fill opacity=1 ][line width=0.75]      (0, 0) circle [x radius= 1.34, y radius= 1.34]   ;
\draw [line width=0.75]    (392.85,187.1) .. controls (464.78,184.05) and (498.22,200.37) .. (566.37,190.61) ;
\draw [shift={(568.44,190.3)}, rotate = 171.5] [fill={rgb, 255:red, 0; green, 0; blue, 0 }  ][line width=0.08]  [draw opacity=0] (5.36,-2.57) -- (0,0) -- (5.36,2.57) -- cycle    ;
\draw [shift={(389.5,187.25)}, rotate = 357.14] [fill={rgb, 255:red, 0; green, 0; blue, 0 }  ][line width=0.08]  [draw opacity=0] (5.36,-2.57) -- (0,0) -- (5.36,2.57) -- cycle    ;
\draw    (483.92,201.73) .. controls (485.89,212.69) and (491.54,212.69) .. (487.88,226.24) ;
\draw [shift={(487.06,229.01)}, rotate = 287.7] [fill={rgb, 255:red, 0; green, 0; blue, 0 }  ][line width=0.08]  [draw opacity=0] (5.36,-2.57) -- (0,0) -- (5.36,2.57) -- cycle    ;
\draw [line width=0.75]    (485.75,140.25) .. controls (481.02,124.63) and (505.87,123.55) .. (480.24,98.94) ;
\draw [shift={(478.17,97)}, rotate = 42.42] [fill={rgb, 255:red, 0; green, 0; blue, 0 }  ][line width=0.08]  [draw opacity=0] (3.57,-1.72) -- (0,0) -- (3.57,1.72) -- cycle    ;
\draw    (481.42,183.5) -- (486.75,192.25) ;
\draw    (486.75,192.25) -- (483.92,201.73) ;
\draw [color={rgb, 255:red, 0; green, 0; blue, 0 }  ,draw opacity=1 ]   (490.75,113.5) .. controls (446.75,132.5) and (463.5,150.5) .. (481.42,183.5) ;
\draw    (467.75,140.75) .. controls (474.84,136.12) and (481.45,142.32) .. (493.95,141.37) .. controls (506.45,140.42) and (528.48,140.27) .. (520.01,193.47) ;
\draw [color={rgb, 255:red, 208; green, 2; blue, 27 }  ,draw opacity=1 ][line width=2.25]    (566.52,190.96) -- (545.1,192.97) -- (520.01,193.47) -- (521.59,179.32) -- (522.17,169.65) -- (520.42,157) -- (515.46,147.57) -- (506.7,142.11) -- (497.36,140.63) -- (490.75,141.25) -- (482.25,139.75) -- (477.75,138.5) -- (472.75,138.75) -- (465.91,140.38) ;
\draw [shift={(571.5,190.49)}, rotate = 174.63] [fill={rgb, 255:red, 208; green, 2; blue, 27 }  ,fill opacity=1 ][line width=0.08]  [draw opacity=0] (11.43,-5.49) -- (0,0) -- (11.43,5.49) -- cycle    ;
\draw [color={rgb, 255:red, 208; green, 2; blue, 27 }  ,draw opacity=1 ][line width=2.25]    (457.94,146.33) -- (447.74,152.78) -- (441.91,159.73) -- (438.69,165.93) -- (439.57,174.36) -- (440.74,180.81) -- (439.91,187.98) -- (427.6,186.77) -- (410.84,186.27) -- (403,186.5) -- (391.74,187.02) ;
\draw [shift={(386.75,187.25)}, rotate = 357.36] [fill={rgb, 255:red, 208; green, 2; blue, 27 }  ,fill opacity=1 ][line width=0.08]  [draw opacity=0] (11.43,-5.49) -- (0,0) -- (11.43,5.49) -- cycle    ;
\draw [color={rgb, 255:red, 74; green, 144; blue, 226 }  ,draw opacity=1 ][line width=1.5]    (477.83,97.08) -- (487,106.25) -- (490.75,113.5) -- (479.25,118.5) -- (471,124.25) -- (467,129) -- (464,134.25) -- (462.75,140) -- (462.75,143.75) -- (464.75,153) -- (470,163) -- (476,173.75) -- (483.75,187.75) -- (486.75,192.25) -- (483.92,201.73) -- (484.25,205.5) -- (486.75,209.75) -- (489.75,217) -- (487.77,227.08) ;
\draw [shift={(487,231)}, rotate = 281.11] [fill={rgb, 255:red, 74; green, 144; blue, 226 }  ,fill opacity=1 ][line width=0.08]  [draw opacity=0] (9.29,-4.46) -- (0,0) -- (9.29,4.46) -- cycle    ;
\draw [shift={(475,94.25)}, rotate = 45] [fill={rgb, 255:red, 74; green, 144; blue, 226 }  ,fill opacity=1 ][line width=0.08]  [draw opacity=0] (9.29,-4.46) -- (0,0) -- (9.29,4.46) -- cycle    ;

\draw (196,147.9) node [anchor=north west][inner sep=0.75pt]  [font=\footnotesize]  {$\ell$};
\draw (266.67,91.15) node [anchor=north west][inner sep=0.75pt]  [font=\footnotesize]  {$w'_{1}$};
\draw (177.25,187.9) node [anchor=north west][inner sep=0.75pt]  [font=\footnotesize]  {$\rho $};
\draw (223,115.9) node [anchor=north west][inner sep=0.75pt]  [font=\footnotesize,color={rgb, 255:red, 208; green, 2; blue, 27 }  ,opacity=1 ]  {$b$};
\draw (404.75,162.65) node [anchor=north west][inner sep=0.75pt]  [font=\footnotesize,color={rgb, 255:red, 208; green, 2; blue, 27 }  ,opacity=1 ]  {$\hat{\rho }$};
\draw (267.17,209.4) node [anchor=north west][inner sep=0.75pt]  [font=\footnotesize]  {$w_{2}$};
\draw (341.5,136.65) node [anchor=north west][inner sep=0.75pt]    {$\Longrightarrow $};
\draw (295.51,192.87) node [anchor=north west][inner sep=0.75pt]  [font=\footnotesize]  {$x$};
\draw (214.76,188.87) node [anchor=north west][inner sep=0.75pt]  [font=\footnotesize]  {$y$};

\end{tikzpicture}

%% file: figures/An-bounded-case-2.tex
\tikzset{every picture/.style={line width=0.75pt}} 

\begin{tikzpicture}[x=0.75pt,y=0.75pt,yscale=-1,xscale=1]

\draw [color={rgb, 255:red, 155; green, 155; blue, 155 }  ,draw opacity=1 ] [dash pattern={on 2.25pt off 2.25pt}]  (191.87,158.25) .. controls (191.87,136.5) and (171.5,179) .. (177.75,136.75) .. controls (184,94.5) and (199.37,163.5) .. (199.52,129.77) ;
\draw [line width=0.75]    (97.2,179.74) .. controls (160.32,175.07) and (215.4,193.32) .. (284.4,183.46) ;
\draw [shift={(286.5,183.15)}, rotate = 171.44] [fill={rgb, 255:red, 0; green, 0; blue, 0 }  ][line width=0.08]  [draw opacity=0] (5.36,-2.57) -- (0,0) -- (5.36,2.57) -- cycle    ;
\draw [shift={(93.29,180.06)}, rotate = 354.92] [fill={rgb, 255:red, 0; green, 0; blue, 0 }  ][line width=0.08]  [draw opacity=0] (5.36,-2.57) -- (0,0) -- (5.36,2.57) -- cycle    ;
\draw    (176.89,182.08) .. controls (164.39,160.36) and (191.1,157.35) .. (204.81,156.92) .. controls (218.51,156.49) and (232.38,167.88) .. (217.65,185.95) ;
\draw    (161.82,145.95) .. controls (166.44,133.26) and (162.33,136.27) .. (191.28,131.33) .. controls (220.22,126.38) and (246.08,132.19) .. (237.52,186.38) ;
\draw    (156.51,180.79) .. controls (159.08,164.66) and (151.64,170.16) .. (159,154.25) ;
\draw    (195.32,184.1) .. controls (191.53,200.93) and (203.94,203.32) .. (200.83,217.4) ;
\draw [shift={(200.03,220.25)}, rotate = 288.44] [fill={rgb, 255:red, 0; green, 0; blue, 0 }  ][line width=0.08]  [draw opacity=0] (5.36,-2.57) -- (0,0) -- (5.36,2.57) -- cycle    ;
\draw    (203.62,176) .. controls (210.87,159) and (190.62,179.5) .. (191.87,158.25) ;
\draw [shift={(191.87,158.25)}, rotate = 273.37] [color={rgb, 255:red, 0; green, 0; blue, 0 }  ][fill={rgb, 255:red, 0; green, 0; blue, 0 }  ][line width=0.75]      (0, 0) circle [x radius= 1.34, y radius= 1.34]   ;
\draw [shift={(203.62,176)}, rotate = 293.1] [color={rgb, 255:red, 0; green, 0; blue, 0 }  ][fill={rgb, 255:red, 0; green, 0; blue, 0 }  ][line width=0.75]      (0, 0) circle [x radius= 1.34, y radius= 1.34]   ;
\draw    (199.52,129.77) .. controls (204.25,114.15) and (179.17,115.9) .. (204.79,91.43) ;
\draw [shift={(206.87,89.5)}, rotate = 137.58] [fill={rgb, 255:red, 0; green, 0; blue, 0 }  ][line width=0.08]  [draw opacity=0] (5.36,-2.57) -- (0,0) -- (5.36,2.57) -- cycle    ;
\draw    (203.62,176) -- (195.65,184.1) ;
\draw [shift={(195.65,184.1)}, rotate = 134.53] [color={rgb, 255:red, 0; green, 0; blue, 0 }  ][fill={rgb, 255:red, 0; green, 0; blue, 0 }  ][line width=0.75]      (0, 0) circle [x radius= 1.34, y radius= 1.34]   ;
\draw [line width=0.75]    (350.45,183.24) .. controls (413.57,178.57) and (468.65,196.82) .. (537.65,186.96) ;
\draw [shift={(539.75,186.65)}, rotate = 171.44] [fill={rgb, 255:red, 0; green, 0; blue, 0 }  ][line width=0.08]  [draw opacity=0] (5.36,-2.57) -- (0,0) -- (5.36,2.57) -- cycle    ;
\draw [shift={(346.54,183.56)}, rotate = 354.92] [fill={rgb, 255:red, 0; green, 0; blue, 0 }  ][line width=0.08]  [draw opacity=0] (5.36,-2.57) -- (0,0) -- (5.36,2.57) -- cycle    ;
\draw    (430.14,185.58) .. controls (417.64,163.86) and (444.35,160.85) .. (458.06,160.42) .. controls (471.76,159.99) and (485.63,171.38) .. (470.9,189.45) ;
\draw    (415.07,149.45) .. controls (419.69,136.76) and (415.58,139.77) .. (444.53,134.83) .. controls (473.47,129.88) and (499.33,135.69) .. (490.77,189.88) ;
\draw    (409.76,184.29) .. controls (412.33,168.16) and (405.62,172.43) .. (412.98,156.52) ;
\draw    (448.57,187.6) .. controls (444.78,204.43) and (457.19,206.82) .. (454.08,220.9) ;
\draw [shift={(453.28,223.75)}, rotate = 288.44] [fill={rgb, 255:red, 0; green, 0; blue, 0 }  ][line width=0.08]  [draw opacity=0] (5.36,-2.57) -- (0,0) -- (5.36,2.57) -- cycle    ;
\draw    (456.53,179.5) .. controls (463.78,162.5) and (443.53,183) .. (444.78,161.75) ;
\draw [shift={(444.78,161.75)}, rotate = 273.37] [color={rgb, 255:red, 0; green, 0; blue, 0 }  ][fill={rgb, 255:red, 0; green, 0; blue, 0 }  ][line width=0.75]      (0, 0) circle [x radius= 1.34, y radius= 1.34]   ;
\draw [shift={(456.53,179.5)}, rotate = 293.1] [color={rgb, 255:red, 0; green, 0; blue, 0 }  ][fill={rgb, 255:red, 0; green, 0; blue, 0 }  ][line width=0.75]      (0, 0) circle [x radius= 1.34, y radius= 1.34]   ;
\draw    (452.43,133.27) .. controls (457.16,117.65) and (432.09,119.4) .. (457.71,94.93) ;
\draw [shift={(459.78,93)}, rotate = 137.58] [fill={rgb, 255:red, 0; green, 0; blue, 0 }  ][line width=0.08]  [draw opacity=0] (5.36,-2.57) -- (0,0) -- (5.36,2.57) -- cycle    ;
\draw    (456.53,179.5) -- (448.57,187.6) ;
\draw [shift={(448.57,187.6)}, rotate = 134.53] [color={rgb, 255:red, 0; green, 0; blue, 0 }  ][fill={rgb, 255:red, 0; green, 0; blue, 0 }  ][line width=0.75]      (0, 0) circle [x radius= 1.34, y radius= 1.34]   ;
\draw [color={rgb, 255:red, 208; green, 2; blue, 27 }  ,draw opacity=1 ][line width=2.25]    (538.72,187.3) -- (515.91,189.37) -- (490.77,189.88) -- (470.9,189.45) -- (474.96,181.79) -- (477,176.48) -- (476.13,169.15) -- (471.15,163.34) -- (462.38,160.31) -- (454.63,159.99) -- (444.78,161.75) -- (444.75,166.25) -- (446.25,170.5) -- (449.75,172.25) -- (454.5,172.75) -- (456.75,173.25) -- (458.5,175.25) -- (456.53,179.5) -- (447,177) -- (439.5,173.5) -- (431.25,163.75) -- (422.5,158.5) -- (412.53,152.5) -- (402.53,150.25) -- (392.03,147) -- (386.75,146.75) -- (382.53,149.75) -- (380.78,157) -- (380.53,165.25) -- (381.06,173.45) -- (381.36,182.55) -- (358.25,182.8) -- (349.78,183.28) ;
\draw [shift={(344.78,183.56)}, rotate = 356.78] [fill={rgb, 255:red, 208; green, 2; blue, 27 }  ,fill opacity=1 ][line width=0.08]  [draw opacity=0] (11.43,-5.49) -- (0,0) -- (11.43,5.49) -- cycle    ;
\draw [shift={(543.7,186.84)}, rotate = 174.8] [fill={rgb, 255:red, 208; green, 2; blue, 27 }  ,fill opacity=1 ][line width=0.08]  [draw opacity=0] (11.43,-5.49) -- (0,0) -- (11.43,5.49) -- cycle    ;
\draw [color={rgb, 255:red, 74; green, 144; blue, 226 }  ,draw opacity=1 ][line width=1.5]    (454.06,221.78) -- (454.78,216.25) -- (453.53,210.25) -- (450.03,202.75) -- (447.53,195.75) -- (448.57,187.6) -- (434,186.59) -- (409.76,184.29) -- (410.89,176.73) -- (409.72,171.43) -- (409.43,166.37) -- (412.98,156.52) ;
\draw [shift={(453.53,225.75)}, rotate = 277.5] [fill={rgb, 255:red, 74; green, 144; blue, 226 }  ,fill opacity=1 ][line width=0.08]  [draw opacity=0] (9.29,-4.46) -- (0,0) -- (9.29,4.46) -- cycle    ;
\draw [color={rgb, 255:red, 74; green, 144; blue, 226 }  ,draw opacity=1 ][line width=1.5]    (458.28,94.15) -- (451.28,101.5) -- (445.78,109.5) -- (446.28,116) -- (451.53,124) -- (453.53,129) -- (452.43,133.27) -- (438.98,136.05) -- (424.06,138.07) -- (418.79,140.34) -- (415.07,149.45) ;
\draw [shift={(461.03,91.25)}, rotate = 133.57] [fill={rgb, 255:red, 74; green, 144; blue, 226 }  ,fill opacity=1 ][line width=0.08]  [draw opacity=0] (9.29,-4.46) -- (0,0) -- (9.29,4.46) -- cycle    ;
\draw [color={rgb, 255:red, 208; green, 2; blue, 27 }  ,draw opacity=1 ]   (128.11,179.05) .. controls (122,154.5) and (120.57,145.75) .. (138.78,143.5) .. controls (157,141.25) and (184.75,170.75) .. (203.62,176) ;
\draw [shift={(203.62,176)}, rotate = 15.55] [color={rgb, 255:red, 208; green, 2; blue, 27 }  ,draw opacity=1 ][fill={rgb, 255:red, 208; green, 2; blue, 27 }  ,fill opacity=1 ][line width=0.75]      (0, 0) circle [x radius= 2.01, y radius= 2.01]   ;
\draw [shift={(128.11,179.05)}, rotate = 256.03] [color={rgb, 255:red, 208; green, 2; blue, 27 }  ,draw opacity=1 ][fill={rgb, 255:red, 208; green, 2; blue, 27 }  ,fill opacity=1 ][line width=0.75]      (0, 0) circle [x radius= 2.01, y radius= 2.01]   ;

\draw (118.78,126.9) node [anchor=north west][inner sep=0.75pt]  [color={rgb, 255:red, 208; green, 2; blue, 27 }  ,opacity=1 ]  {$b$};
\draw (211.53,168.65) node [anchor=north west][inner sep=0.75pt]  [font=\tiny]  {$q_1$};
\draw (244.03,142.4) node [anchor=north west][inner sep=0.75pt]  [font=\footnotesize]  {$\ell$};
\draw (209.2,79.65) node [anchor=north west][inner sep=0.75pt]  [font=\footnotesize]  {$w'_{1}$};
\draw (209.28,208.4) node [anchor=north west][inner sep=0.75pt]  [font=\footnotesize]  {$w_{2}$};
\draw (300.03,141.9) node [anchor=north west][inner sep=0.75pt]    {$\Longrightarrow $};
\draw (525.25,195.95) node [anchor=north west][inner sep=0.75pt]  [color={rgb, 255:red, 208; green, 2; blue, 27 }  ,opacity=1 ]  {$\hat{\rho }$};

\end{tikzpicture}

%% file: figures/crosshair-final.tex
\tikzset{every picture/.style={line width=0.75pt}} 

\begin{tikzpicture}[x=0.75pt,y=0.75pt,yscale=-1,xscale=1]

\draw [color={rgb, 255:red, 155; green, 155; blue, 155 }  ,draw opacity=1 ]   (295.71,157.05) -- (320.88,145.15) ;
\draw [shift={(293,158.33)}, rotate = 334.69] [fill={rgb, 255:red, 155; green, 155; blue, 155 }  ,fill opacity=1 ][line width=0.08]  [draw opacity=0] (3.57,-1.72) -- (0,0) -- (3.57,1.72) -- cycle    ;
\draw [color={rgb, 255:red, 155; green, 155; blue, 155 }  ,draw opacity=1 ]   (391.69,127.42) -- (320.88,145.15) ;
\draw [shift={(394.6,126.69)}, rotate = 165.94] [fill={rgb, 255:red, 155; green, 155; blue, 155 }  ,fill opacity=1 ][line width=0.08]  [draw opacity=0] (3.57,-1.72) -- (0,0) -- (3.57,1.72) -- cycle    ;
\draw    (177.31,48.22) -- (320.88,145.15) ;
\draw [shift={(320.88,145.15)}, rotate = 34.03] [color={rgb, 255:red, 0; green, 0; blue, 0 }  ][fill={rgb, 255:red, 0; green, 0; blue, 0 }  ][line width=0.75]      (0, 0) circle [x radius= 2.68, y radius= 2.68]   ;
\draw [shift={(174.82,46.54)}, rotate = 34.03] [fill={rgb, 255:red, 0; green, 0; blue, 0 }  ][line width=0.08]  [draw opacity=0] (7.14,-3.43) -- (0,0) -- (7.14,3.43) -- cycle    ;
\draw    (464.77,50.15) -- (320.88,145.15) ;
\draw [shift={(320.88,145.15)}, rotate = 146.56] [color={rgb, 255:red, 0; green, 0; blue, 0 }  ][fill={rgb, 255:red, 0; green, 0; blue, 0 }  ][line width=0.75]      (0, 0) circle [x radius= 2.68, y radius= 2.68]   ;
\draw [shift={(467.27,48.5)}, rotate = 146.56] [fill={rgb, 255:red, 0; green, 0; blue, 0 }  ][line width=0.08]  [draw opacity=0] (7.14,-3.43) -- (0,0) -- (7.14,3.43) -- cycle    ;
\draw    (308.99,259.07) -- (320.88,145.15) ;
\draw [shift={(320.88,145.15)}, rotate = 275.96] [color={rgb, 255:red, 0; green, 0; blue, 0 }  ][fill={rgb, 255:red, 0; green, 0; blue, 0 }  ][line width=0.75]      (0, 0) circle [x radius= 2.68, y radius= 2.68]   ;
\draw [shift={(308.68,262.05)}, rotate = 275.96] [fill={rgb, 255:red, 0; green, 0; blue, 0 }  ][line width=0.08]  [draw opacity=0] (7.14,-3.43) -- (0,0) -- (7.14,3.43) -- cycle    ;
\draw [color={rgb, 255:red, 65; green, 117; blue, 5 }  ,draw opacity=1 ]   (169.17,161.29) .. controls (231.54,123.34) and (200.53,210) .. (308.02,137.31) ;
\draw [shift={(308.02,137.31)}, rotate = 325.93] [color={rgb, 255:red, 65; green, 117; blue, 5 }  ,draw opacity=1 ][fill={rgb, 255:red, 65; green, 117; blue, 5 }  ,fill opacity=1 ][line width=0.75]      (0, 0) circle [x radius= 2.01, y radius= 2.01]   ;
\draw [shift={(166.25,163.11)}, rotate = 327.36] [fill={rgb, 255:red, 65; green, 117; blue, 5 }  ,fill opacity=1 ][line width=0.08]  [draw opacity=0] (5.36,-2.57) -- (0,0) -- (5.36,2.57) -- cycle    ;
\draw [color={rgb, 255:red, 65; green, 117; blue, 5 }  ,draw opacity=1 ]   (488.64,198.38) .. controls (389.83,124.52) and (391.69,186.03) .. (318.9,159.52) ;
\draw [shift={(318.9,159.52)}, rotate = 200.01] [color={rgb, 255:red, 65; green, 117; blue, 5 }  ,draw opacity=1 ][fill={rgb, 255:red, 65; green, 117; blue, 5 }  ,fill opacity=1 ][line width=0.75]      (0, 0) circle [x radius= 2.01, y radius= 2.01]   ;
\draw [shift={(491.67,200.66)}, rotate = 217.22] [fill={rgb, 255:red, 65; green, 117; blue, 5 }  ,fill opacity=1 ][line width=0.08]  [draw opacity=0] (5.36,-2.57) -- (0,0) -- (5.36,2.57) -- cycle    ;
\draw [color={rgb, 255:red, 65; green, 117; blue, 5 }  ,draw opacity=1 ]   (307.94,56.47) .. controls (348.26,95.17) and (316.85,95.13) .. (329.12,139.93) ;
\draw [shift={(329.12,139.93)}, rotate = 74.67] [color={rgb, 255:red, 65; green, 117; blue, 5 }  ,draw opacity=1 ][fill={rgb, 255:red, 65; green, 117; blue, 5 }  ,fill opacity=1 ][line width=0.75]      (0, 0) circle [x radius= 2.01, y radius= 2.01]   ;
\draw [shift={(305.38,54.05)}, rotate = 42.96] [fill={rgb, 255:red, 65; green, 117; blue, 5 }  ,fill opacity=1 ][line width=0.08]  [draw opacity=0] (5.36,-2.57) -- (0,0) -- (5.36,2.57) -- cycle    ;
\draw [color={rgb, 255:red, 208; green, 2; blue, 27 }  ,draw opacity=1 ]   (270.44,111.19) .. controls (297.8,85.39) and (347.26,95.84) .. (362.42,118.7) ;
\draw [shift={(362.42,118.7)}, rotate = 56.43] [color={rgb, 255:red, 208; green, 2; blue, 27 }  ,draw opacity=1 ][fill={rgb, 255:red, 208; green, 2; blue, 27 }  ,fill opacity=1 ][line width=0.75]      (0, 0) circle [x radius= 1.34, y radius= 1.34]   ;
\draw [shift={(270.44,111.19)}, rotate = 316.69] [color={rgb, 255:red, 208; green, 2; blue, 27 }  ,draw opacity=1 ][fill={rgb, 255:red, 208; green, 2; blue, 27 }  ,fill opacity=1 ][line width=0.75]      (0, 0) circle [x radius= 1.34, y radius= 1.34]   ;
\draw [color={rgb, 255:red, 208; green, 2; blue, 27 }  ,draw opacity=1 ]   (352.53,124.91) .. controls (370.67,145.48) and (372.31,181.07) .. (315.61,189.23) ;
\draw [shift={(315.61,189.23)}, rotate = 171.81] [color={rgb, 255:red, 208; green, 2; blue, 27 }  ,draw opacity=1 ][fill={rgb, 255:red, 208; green, 2; blue, 27 }  ,fill opacity=1 ][line width=0.75]      (0, 0) circle [x radius= 1.34, y radius= 1.34]   ;
\draw [shift={(352.53,124.91)}, rotate = 48.6] [color={rgb, 255:red, 208; green, 2; blue, 27 }  ,draw opacity=1 ][fill={rgb, 255:red, 208; green, 2; blue, 27 }  ,fill opacity=1 ][line width=0.75]      (0, 0) circle [x radius= 1.34, y radius= 1.34]   ;
\draw [color={rgb, 255:red, 208; green, 2; blue, 27 }  ,draw opacity=1 ]   (283.62,120.66) .. controls (268.79,142.54) and (272.74,189.89) .. (314.62,204.58) ;
\draw [shift={(314.62,204.58)}, rotate = 19.34] [color={rgb, 255:red, 208; green, 2; blue, 27 }  ,draw opacity=1 ][fill={rgb, 255:red, 208; green, 2; blue, 27 }  ,fill opacity=1 ][line width=0.75]      (0, 0) circle [x radius= 1.34, y radius= 1.34]   ;
\draw [shift={(283.62,120.66)}, rotate = 124.14] [color={rgb, 255:red, 208; green, 2; blue, 27 }  ,draw opacity=1 ][fill={rgb, 255:red, 208; green, 2; blue, 27 }  ,fill opacity=1 ][line width=0.75]      (0, 0) circle [x radius= 1.34, y radius= 1.34]   ;
\draw  [color={rgb, 255:red, 155; green, 155; blue, 155 }  ,draw opacity=1 ][dash pattern={on 2.25pt off 1.5pt}] (289.71,145.15) .. controls (289.71,127.93) and (303.66,113.98) .. (320.88,113.98) .. controls (338.1,113.98) and (352.05,127.93) .. (352.05,145.15) .. controls (352.05,162.37) and (338.1,176.32) .. (320.88,176.32) .. controls (303.66,176.32) and (289.71,162.37) .. (289.71,145.15) -- cycle ;
\draw  [color={rgb, 255:red, 155; green, 155; blue, 155 }  ,draw opacity=1 ][dash pattern={on 2.25pt off 1.5pt}] (242.98,145.15) .. controls (242.98,102.13) and (277.86,67.25) .. (320.88,67.25) .. controls (363.9,67.25) and (398.78,102.13) .. (398.78,145.15) .. controls (398.78,188.17) and (363.9,223.05) .. (320.88,223.05) .. controls (277.86,223.05) and (242.98,188.17) .. (242.98,145.15) -- cycle ;

\draw (268.06,182.5) node [anchor=north west][inner sep=0.75pt]  [font=\scriptsize,color={rgb, 255:red, 208; green, 2; blue, 27 }  ,opacity=1 ]  {$q_{1}$};
\draw (340.93,186.09) node [anchor=north west][inner sep=0.75pt]  [font=\scriptsize,color={rgb, 255:red, 208; green, 2; blue, 27 }  ,opacity=1 ]  {$q_{2}$};
\draw (294.79,81.06) node [anchor=north west][inner sep=0.75pt]  [font=\scriptsize,color={rgb, 255:red, 208; green, 2; blue, 27 }  ,opacity=1 ]  {$q_{3}$};
\draw (323.6,50.91) node [anchor=north west][inner sep=0.75pt]  [font=\scriptsize,color={rgb, 255:red, 65; green, 117; blue, 5 }  ,opacity=1 ]  {$w_{3}$};
\draw (450.21,154.42) node [anchor=north west][inner sep=0.75pt]  [font=\scriptsize,color={rgb, 255:red, 65; green, 117; blue, 5 }  ,opacity=1 ]  {$w_{2}$};
\draw (187.43,159.97) node [anchor=north west][inner sep=0.75pt]  [font=\scriptsize,color={rgb, 255:red, 65; green, 117; blue, 5 }  ,opacity=1 ]  {$w_{1}$};
\draw (173.93,62.99) node [anchor=north west][inner sep=0.75pt]  [font=\scriptsize,color={rgb, 255:red, 0; green, 0; blue, 0 }  ,opacity=1 ]  {$\gamma _{2}$};
\draw (455.49,65.27) node [anchor=north west][inner sep=0.75pt]  [font=\scriptsize,color={rgb, 255:red, 0; green, 0; blue, 0 }  ,opacity=1 ]  {$\gamma _{1}$};
\draw (286.03,232.13) node [anchor=north west][inner sep=0.75pt]  [font=\scriptsize,color={rgb, 255:red, 0; green, 0; blue, 0 }  ,opacity=1 ]  {$\gamma _{3}$};
\draw (301.32,162.31) node [anchor=north west][inner sep=0.75pt]  [font=\tiny,color={rgb, 255:red, 155; green, 155; blue, 155 }  ,opacity=1 ]  {$r$};
\draw (385.12,139.56) node [anchor=north west][inner sep=0.75pt]  [font=\tiny,color={rgb, 255:red, 155; green, 155; blue, 155 }  ,opacity=1 ]  {$R$};

\end{tikzpicture}

%% file: figures/crosshair-bad-case.tex
\tikzset{every picture/.style={line width=0.75pt}} 

\begin{tikzpicture}[x=0.75pt,y=0.75pt,yscale=-1,xscale=1]

\draw [color={rgb, 255:red, 0; green, 0; blue, 0 }  ,draw opacity=1 ]   (94.45,50.39) -- (196.49,122.24) ;
\draw [shift={(196.49,122.24)}, rotate = 35.15] [color={rgb, 255:red, 0; green, 0; blue, 0 }  ,draw opacity=1 ][fill={rgb, 255:red, 0; green, 0; blue, 0 }  ,fill opacity=1 ][line width=0.75]      (0, 0) circle [x radius= 2.68, y radius= 2.68]   ;
\draw [shift={(92,48.67)}, rotate = 35.15] [fill={rgb, 255:red, 0; green, 0; blue, 0 }  ,fill opacity=1 ][line width=0.08]  [draw opacity=0] (7.14,-3.43) -- (0,0) -- (7.14,3.43) -- cycle    ;
\draw [color={rgb, 255:red, 0; green, 0; blue, 0 }  ,draw opacity=1 ]   (292.86,55.7) -- (196.49,122.24) ;
\draw [shift={(196.49,122.24)}, rotate = 145.38] [color={rgb, 255:red, 0; green, 0; blue, 0 }  ,draw opacity=1 ][fill={rgb, 255:red, 0; green, 0; blue, 0 }  ,fill opacity=1 ][line width=0.75]      (0, 0) circle [x radius= 2.68, y radius= 2.68]   ;
\draw [shift={(295.33,54)}, rotate = 145.38] [fill={rgb, 255:red, 0; green, 0; blue, 0 }  ,fill opacity=1 ][line width=0.08]  [draw opacity=0] (7.14,-3.43) -- (0,0) -- (7.14,3.43) -- cycle    ;
\draw [color={rgb, 255:red, 0; green, 0; blue, 0 }  ,draw opacity=1 ]   (189.56,214.01) -- (196.49,122.24) ;
\draw [shift={(196.49,122.24)}, rotate = 274.32] [color={rgb, 255:red, 0; green, 0; blue, 0 }  ,draw opacity=1 ][fill={rgb, 255:red, 0; green, 0; blue, 0 }  ,fill opacity=1 ][line width=0.75]      (0, 0) circle [x radius= 2.68, y radius= 2.68]   ;
\draw [shift={(189.33,217)}, rotate = 274.32] [fill={rgb, 255:red, 0; green, 0; blue, 0 }  ,fill opacity=1 ][line width=0.08]  [draw opacity=0] (7.14,-3.43) -- (0,0) -- (7.14,3.43) -- cycle    ;
\draw [color={rgb, 255:red, 74; green, 144; blue, 226 }  ,draw opacity=1 ][line width=1.5]    (171,104) .. controls (188.33,87.67) and (206,88) .. (220.67,104.67) ;
\draw [shift={(220.67,104.67)}, rotate = 48.65] [color={rgb, 255:red, 74; green, 144; blue, 226 }  ,draw opacity=1 ][fill={rgb, 255:red, 74; green, 144; blue, 226 }  ,fill opacity=1 ][line width=1.5]      (0, 0) circle [x radius= 1.74, y radius= 1.74]   ;
\draw [shift={(171,104)}, rotate = 316.7] [color={rgb, 255:red, 74; green, 144; blue, 226 }  ,draw opacity=1 ][fill={rgb, 255:red, 74; green, 144; blue, 226 }  ,fill opacity=1 ][line width=1.5]      (0, 0) circle [x radius= 1.74, y radius= 1.74]   ;
\draw [color={rgb, 255:red, 74; green, 144; blue, 226 }  ,draw opacity=1 ][line width=1.5]    (140.67,83) .. controls (173,56.33) and (215,68) .. (249.33,85) ;
\draw [shift={(249.33,85)}, rotate = 26.34] [color={rgb, 255:red, 74; green, 144; blue, 226 }  ,draw opacity=1 ][fill={rgb, 255:red, 74; green, 144; blue, 226 }  ,fill opacity=1 ][line width=1.5]      (0, 0) circle [x radius= 1.74, y radius= 1.74]   ;
\draw [shift={(140.67,83)}, rotate = 320.49] [color={rgb, 255:red, 74; green, 144; blue, 226 }  ,draw opacity=1 ][fill={rgb, 255:red, 74; green, 144; blue, 226 }  ,fill opacity=1 ][line width=1.5]      (0, 0) circle [x radius= 1.74, y radius= 1.74]   ;
\draw [color={rgb, 255:red, 74; green, 144; blue, 226 }  ,draw opacity=1 ][line width=1.5]    (230,101.33) .. controls (237.67,127) and (232.67,141) .. (192.82,145.94) ;
\draw [shift={(192.82,145.94)}, rotate = 172.94] [color={rgb, 255:red, 74; green, 144; blue, 226 }  ,draw opacity=1 ][fill={rgb, 255:red, 74; green, 144; blue, 226 }  ,fill opacity=1 ][line width=1.5]      (0, 0) circle [x radius= 1.74, y radius= 1.74]   ;
\draw [shift={(230,101.33)}, rotate = 73.37] [color={rgb, 255:red, 74; green, 144; blue, 226 }  ,draw opacity=1 ][fill={rgb, 255:red, 74; green, 144; blue, 226 }  ,fill opacity=1 ][line width=1.5]      (0, 0) circle [x radius= 1.74, y radius= 1.74]   ;
\draw [color={rgb, 255:red, 74; green, 144; blue, 226 }  ,draw opacity=1 ][line width=1.5]    (257.33,81.33) .. controls (278,119.33) and (254.67,170.33) .. (192.63,172.35) ;
\draw [shift={(192.63,172.35)}, rotate = 178.14] [color={rgb, 255:red, 74; green, 144; blue, 226 }  ,draw opacity=1 ][fill={rgb, 255:red, 74; green, 144; blue, 226 }  ,fill opacity=1 ][line width=1.5]      (0, 0) circle [x radius= 1.74, y radius= 1.74]   ;
\draw [shift={(257.33,81.33)}, rotate = 61.46] [color={rgb, 255:red, 74; green, 144; blue, 226 }  ,draw opacity=1 ][fill={rgb, 255:red, 74; green, 144; blue, 226 }  ,fill opacity=1 ][line width=1.5]      (0, 0) circle [x radius= 1.74, y radius= 1.74]   ;
\draw [color={rgb, 255:red, 155; green, 155; blue, 155 }  ,draw opacity=1 ]   (117.07,143.38) .. controls (136.83,130.63) and (138.81,166.4) .. (188.17,117.45) ;
\draw [shift={(188.17,117.45)}, rotate = 315.23] [color={rgb, 255:red, 155; green, 155; blue, 155 }  ,draw opacity=1 ][fill={rgb, 255:red, 155; green, 155; blue, 155 }  ,fill opacity=1 ][line width=0.75]      (0, 0) circle [x radius= 2.01, y radius= 2.01]   ;
\draw [shift={(114.5,145.24)}, rotate = 321.44] [fill={rgb, 255:red, 155; green, 155; blue, 155 }  ,fill opacity=1 ][line width=0.08]  [draw opacity=0] (5.36,-2.57) -- (0,0) -- (5.36,2.57) -- cycle    ;
\draw [color={rgb, 255:red, 155; green, 155; blue, 155 }  ,draw opacity=1 ]   (290,135.81) .. controls (236.33,112.34) and (237.76,138.82) .. (196.45,131.34) ;
\draw [shift={(196.45,131.34)}, rotate = 190.26] [color={rgb, 255:red, 155; green, 155; blue, 155 }  ,draw opacity=1 ][fill={rgb, 255:red, 155; green, 155; blue, 155 }  ,fill opacity=1 ][line width=0.75]      (0, 0) circle [x radius= 2.01, y radius= 2.01]   ;
\draw [shift={(293.35,137.31)}, rotate = 204.53] [fill={rgb, 255:red, 155; green, 155; blue, 155 }  ,fill opacity=1 ][line width=0.08]  [draw opacity=0] (5.36,-2.57) -- (0,0) -- (5.36,2.57) -- cycle    ;
\draw [color={rgb, 255:red, 155; green, 155; blue, 155 }  ,draw opacity=1 ]   (177.06,54.47) .. controls (196.45,68.73) and (190.25,95.42) .. (201.02,118.35) ;
\draw [shift={(201.02,118.35)}, rotate = 64.84] [color={rgb, 255:red, 155; green, 155; blue, 155 }  ,draw opacity=1 ][fill={rgb, 255:red, 155; green, 155; blue, 155 }  ,fill opacity=1 ][line width=0.75]      (0, 0) circle [x radius= 2.01, y radius= 2.01]   ;
\draw [shift={(174.5,52.75)}, rotate = 31.43] [fill={rgb, 255:red, 155; green, 155; blue, 155 }  ,fill opacity=1 ][line width=0.08]  [draw opacity=0] (5.36,-2.57) -- (0,0) -- (5.36,2.57) -- cycle    ;

\end{tikzpicture}

%% file: figures/crosshair-cheap-fix.tex
\tikzset{every picture/.style={line width=0.75pt}} 

\begin{tikzpicture}[x=0.75pt,y=0.75pt,yscale=-1,xscale=1]

\draw [color={rgb, 255:red, 155; green, 155; blue, 155 }  ,draw opacity=1 ] [dash pattern={on 2.25pt off 1.5pt}]  (231.25,130) .. controls (216.75,126.75) and (207.25,133.75) .. (196.45,131.34) ;
\draw [color={rgb, 255:red, 155; green, 155; blue, 155 }  ,draw opacity=1 ] [dash pattern={on 2.25pt off 1.5pt}]  (194.25,92) .. controls (194,104.75) and (194,109) .. (200.77,118.35) ;
\draw [color={rgb, 255:red, 0; green, 0; blue, 0 }  ,draw opacity=1 ][line width=1.5]    (189.99,216.51) -- (194.25,145.75) -- (200.5,145.25) -- (211.75,142.25) -- (221.25,138.25) -- (231.5,130) -- (233,117.75) -- (231.5,106) -- (229.25,99.25) ;
\draw [shift={(229.25,99.25)}, rotate = 251.57] [color={rgb, 255:red, 0; green, 0; blue, 0 }  ,draw opacity=1 ][fill={rgb, 255:red, 0; green, 0; blue, 0 }  ,fill opacity=1 ][line width=1.5]      (0, 0) circle [x radius= 3.05, y radius= 3.05]   ;
\draw [shift={(189.75,220.5)}, rotate = 273.45] [fill={rgb, 255:red, 0; green, 0; blue, 0 }  ,fill opacity=1 ][line width=0.08]  [draw opacity=0] (8.13,-3.9) -- (0,0) -- (8.13,3.9) -- cycle    ;
\draw [color={rgb, 255:red, 0; green, 0; blue, 0 }  ,draw opacity=1 ][line width=1.5]    (92.27,48.31) -- (171,104) -- (179.5,97) -- (188.75,92.25) -- (198.25,91.75) -- (206.5,93.5) -- (212.5,96.75) -- (220.67,104.67) -- (229.25,99.25) ;
\draw [shift={(229.25,99.25)}, rotate = 327.75] [color={rgb, 255:red, 0; green, 0; blue, 0 }  ,draw opacity=1 ][fill={rgb, 255:red, 0; green, 0; blue, 0 }  ,fill opacity=1 ][line width=1.5]      (0, 0) circle [x radius= 3.05, y radius= 3.05]   ;
\draw [shift={(89,46)}, rotate = 35.27] [fill={rgb, 255:red, 0; green, 0; blue, 0 }  ,fill opacity=1 ][line width=0.08]  [draw opacity=0] (8.13,-3.9) -- (0,0) -- (8.13,3.9) -- cycle    ;
\draw [color={rgb, 255:red, 0; green, 0; blue, 0 }  ,draw opacity=1 ][line width=1.5]    (296.2,53.26) -- (229.25,99.25) ;
\draw [shift={(229.25,99.25)}, rotate = 145.52] [color={rgb, 255:red, 0; green, 0; blue, 0 }  ,draw opacity=1 ][fill={rgb, 255:red, 0; green, 0; blue, 0 }  ,fill opacity=1 ][line width=1.5]      (0, 0) circle [x radius= 3.05, y radius= 3.05]   ;
\draw [shift={(299.5,51)}, rotate = 145.52] [fill={rgb, 255:red, 0; green, 0; blue, 0 }  ,fill opacity=1 ][line width=0.08]  [draw opacity=0] (8.13,-3.9) -- (0,0) -- (8.13,3.9) -- cycle    ;
\draw [color={rgb, 255:red, 208; green, 2; blue, 27 }  ,draw opacity=1 ][line width=3]    (140.42,83) .. controls (172.75,56.33) and (214.75,68) .. (249.08,85) ;
\draw [shift={(249.08,85)}, rotate = 26.34] [color={rgb, 255:red, 208; green, 2; blue, 27 }  ,draw opacity=1 ][fill={rgb, 255:red, 208; green, 2; blue, 27 }  ,fill opacity=1 ][line width=3]      (0, 0) circle [x radius= 2.55, y radius= 2.55]   ;
\draw [shift={(140.42,83)}, rotate = 320.49] [color={rgb, 255:red, 208; green, 2; blue, 27 }  ,draw opacity=1 ][fill={rgb, 255:red, 208; green, 2; blue, 27 }  ,fill opacity=1 ][line width=3]      (0, 0) circle [x radius= 2.55, y radius= 2.55]   ;
\draw [color={rgb, 255:red, 208; green, 2; blue, 27 }  ,draw opacity=1 ][line width=3]    (257.08,81.33) .. controls (277.75,119.33) and (254.42,170.33) .. (192.38,172.35) ;
\draw [shift={(192.38,172.35)}, rotate = 178.14] [color={rgb, 255:red, 208; green, 2; blue, 27 }  ,draw opacity=1 ][fill={rgb, 255:red, 208; green, 2; blue, 27 }  ,fill opacity=1 ][line width=3]      (0, 0) circle [x radius= 2.55, y radius= 2.55]   ;
\draw [shift={(257.08,81.33)}, rotate = 61.46] [color={rgb, 255:red, 208; green, 2; blue, 27 }  ,draw opacity=1 ][fill={rgb, 255:red, 208; green, 2; blue, 27 }  ,fill opacity=1 ][line width=3]      (0, 0) circle [x radius= 2.55, y radius= 2.55]   ;
\draw [color={rgb, 255:red, 208; green, 2; blue, 27 }  ,draw opacity=1 ][line width=3]    (170.75,104) -- (196.24,122.24) -- (194,145.75) ;
\draw [shift={(194,145.75)}, rotate = 95.45] [color={rgb, 255:red, 208; green, 2; blue, 27 }  ,draw opacity=1 ][fill={rgb, 255:red, 208; green, 2; blue, 27 }  ,fill opacity=1 ][line width=3]      (0, 0) circle [x radius= 2.55, y radius= 2.55]   ;
\draw [shift={(170.75,104)}, rotate = 35.58] [color={rgb, 255:red, 208; green, 2; blue, 27 }  ,draw opacity=1 ][fill={rgb, 255:red, 208; green, 2; blue, 27 }  ,fill opacity=1 ][line width=3]      (0, 0) circle [x radius= 2.55, y radius= 2.55]   ;
\draw [color={rgb, 255:red, 65; green, 117; blue, 5 }  ,draw opacity=1 ][line width=1.5]    (194.25,92) -- (192.25,81.5) -- (190.5,76.25) -- (188.25,69) -- (183,59.75) -- (176.52,54.12) ;
\draw [shift={(173.5,51.5)}, rotate = 40.97] [fill={rgb, 255:red, 65; green, 117; blue, 5 }  ,fill opacity=1 ][line width=0.08]  [draw opacity=0] (6.97,-3.35) -- (0,0) -- (6.97,3.35) -- cycle    ;
\draw [shift={(194.25,92)}, rotate = 259.22] [color={rgb, 255:red, 65; green, 117; blue, 5 }  ,draw opacity=1 ][fill={rgb, 255:red, 65; green, 117; blue, 5 }  ,fill opacity=1 ][line width=1.5]      (0, 0) circle [x radius= 2.61, y radius= 2.61]   ;
\draw [color={rgb, 255:red, 65; green, 117; blue, 5 }  ,draw opacity=1 ][line width=1.5]    (231.25,130) -- (243,126.75) -- (251.75,125.75) -- (263,126.75) -- (276.5,130.5) -- (290.77,136.05) ;
\draw [shift={(294.5,137.5)}, rotate = 201.25] [fill={rgb, 255:red, 65; green, 117; blue, 5 }  ,fill opacity=1 ][line width=0.08]  [draw opacity=0] (6.97,-3.35) -- (0,0) -- (6.97,3.35) -- cycle    ;
\draw [shift={(231.25,130)}, rotate = 344.54] [color={rgb, 255:red, 65; green, 117; blue, 5 }  ,draw opacity=1 ][fill={rgb, 255:red, 65; green, 117; blue, 5 }  ,fill opacity=1 ][line width=1.5]      (0, 0) circle [x radius= 2.61, y radius= 2.61]   ;
\draw [color={rgb, 255:red, 65; green, 117; blue, 5 }  ,draw opacity=1 ][line width=1.5]    (220.42,104.67) -- (196.24,122.24) -- (188.5,116.75) -- (178,128) -- (167.5,135.25) -- (154.25,142.75) -- (142.75,144.75) -- (132.25,141.75) -- (125.75,140.25) -- (120.25,141.25) -- (116.71,143.87) ;
\draw [shift={(113.5,146.25)}, rotate = 323.47] [fill={rgb, 255:red, 65; green, 117; blue, 5 }  ,fill opacity=1 ][line width=0.08]  [draw opacity=0] (6.97,-3.35) -- (0,0) -- (6.97,3.35) -- cycle    ;
\draw [shift={(220.42,104.67)}, rotate = 143.99] [color={rgb, 255:red, 65; green, 117; blue, 5 }  ,draw opacity=1 ][fill={rgb, 255:red, 65; green, 117; blue, 5 }  ,fill opacity=1 ][line width=1.5]      (0, 0) circle [x radius= 2.61, y radius= 2.61]   ;

\end{tikzpicture}

%% file: figures/witness-separation.tex
\tikzset{every picture/.style={line width=0.75pt}} 

\begin{tikzpicture}[x=0.75pt,y=0.75pt,yscale=-1,xscale=1]

\draw    (189.41,82.48) -- (255.36,127.62) ;
\draw [shift={(255.36,127.62)}, rotate = 34.39] [color={rgb, 255:red, 0; green, 0; blue, 0 }  ][fill={rgb, 255:red, 0; green, 0; blue, 0 }  ][line width=0.75]      (0, 0) circle [x radius= 2.68, y radius= 2.68]   ;
\draw [shift={(186.93,80.78)}, rotate = 34.39] [fill={rgb, 255:red, 0; green, 0; blue, 0 }  ][line width=0.08]  [draw opacity=0] (7.14,-3.43) -- (0,0) -- (7.14,3.43) -- cycle    ;
\draw    (317.82,87.43) -- (255.36,127.62) ;
\draw [shift={(255.36,127.62)}, rotate = 147.24] [color={rgb, 255:red, 0; green, 0; blue, 0 }  ][fill={rgb, 255:red, 0; green, 0; blue, 0 }  ][line width=0.75]      (0, 0) circle [x radius= 2.68, y radius= 2.68]   ;
\draw [shift={(320.34,85.81)}, rotate = 147.24] [fill={rgb, 255:red, 0; green, 0; blue, 0 }  ][line width=0.08]  [draw opacity=0] (7.14,-3.43) -- (0,0) -- (7.14,3.43) -- cycle    ;
\draw    (248.32,195.68) -- (255.36,127.62) ;
\draw [shift={(255.36,127.62)}, rotate = 275.9] [color={rgb, 255:red, 0; green, 0; blue, 0 }  ][fill={rgb, 255:red, 0; green, 0; blue, 0 }  ][line width=0.75]      (0, 0) circle [x radius= 2.68, y radius= 2.68]   ;
\draw [shift={(248.01,198.67)}, rotate = 275.9] [fill={rgb, 255:red, 0; green, 0; blue, 0 }  ][line width=0.08]  [draw opacity=0] (7.14,-3.43) -- (0,0) -- (7.14,3.43) -- cycle    ;
\draw [color={rgb, 255:red, 65; green, 117; blue, 5 }  ,draw opacity=1 ]   (165.13,136.75) .. controls (201.2,115.42) and (183.55,166.59) .. (247.62,122.86) ;
\draw [shift={(247.62,122.86)}, rotate = 325.68] [color={rgb, 255:red, 65; green, 117; blue, 5 }  ,draw opacity=1 ][fill={rgb, 255:red, 65; green, 117; blue, 5 }  ,fill opacity=1 ][line width=0.75]      (0, 0) circle [x radius= 2.01, y radius= 2.01]   ;
\draw [shift={(162.25,138.53)}, rotate = 327.12] [fill={rgb, 255:red, 65; green, 117; blue, 5 }  ,fill opacity=1 ][line width=0.08]  [draw opacity=0] (5.36,-2.57) -- (0,0) -- (5.36,2.57) -- cycle    ;
\draw [color={rgb, 255:red, 65; green, 117; blue, 5 }  ,draw opacity=1 ]   (327.97,155.65) .. controls (285.14,130.38) and (297.55,152.3) .. (254.17,136.35) ;
\draw [shift={(254.17,136.35)}, rotate = 200.18] [color={rgb, 255:red, 65; green, 117; blue, 5 }  ,draw opacity=1 ][fill={rgb, 255:red, 65; green, 117; blue, 5 }  ,fill opacity=1 ][line width=0.75]      (0, 0) circle [x radius= 2.01, y radius= 2.01]   ;
\draw [shift={(330.67,157.26)}, rotate = 211.1] [fill={rgb, 255:red, 65; green, 117; blue, 5 }  ,fill opacity=1 ][line width=0.08]  [draw opacity=0] (5.36,-2.57) -- (0,0) -- (5.36,2.57) -- cycle    ;
\draw [color={rgb, 255:red, 65; green, 117; blue, 5 }  ,draw opacity=1 ]   (248.3,74.43) .. controls (271.46,97.25) and (253,97.49) .. (260.32,124.44) ;
\draw [shift={(260.32,124.44)}, rotate = 74.81] [color={rgb, 255:red, 65; green, 117; blue, 5 }  ,draw opacity=1 ][fill={rgb, 255:red, 65; green, 117; blue, 5 }  ,fill opacity=1 ][line width=0.75]      (0, 0) circle [x radius= 2.01, y radius= 2.01]   ;
\draw [shift={(246.03,72.25)}, rotate = 43.23] [fill={rgb, 255:red, 65; green, 117; blue, 5 }  ,fill opacity=1 ][line width=0.08]  [draw opacity=0] (5.36,-2.57) -- (0,0) -- (5.36,2.57) -- cycle    ;
\draw [color={rgb, 255:red, 65; green, 117; blue, 5 }  ,draw opacity=1 ]   (238.02,115.71) .. controls (249.4,104.6) and (262.11,106.98) .. (271.64,116.24) ;
\draw [shift={(271.64,116.24)}, rotate = 44.18] [color={rgb, 255:red, 65; green, 117; blue, 5 }  ,draw opacity=1 ][fill={rgb, 255:red, 65; green, 117; blue, 5 }  ,fill opacity=1 ][line width=0.75]      (0, 0) circle [x radius= 1.34, y radius= 1.34]   ;
\draw [shift={(238.02,115.71)}, rotate = 315.68] [color={rgb, 255:red, 65; green, 117; blue, 5 }  ,draw opacity=1 ][fill={rgb, 255:red, 65; green, 117; blue, 5 }  ,fill opacity=1 ][line width=0.75]      (0, 0) circle [x radius= 1.34, y radius= 1.34]   ;
\draw [color={rgb, 255:red, 65; green, 117; blue, 5 }  ,draw opacity=1 ]   (267.93,119.95) .. controls (274.02,133.44) and (270.31,143.5) .. (253.64,145.61) ;
\draw [shift={(253.64,145.61)}, rotate = 172.77] [color={rgb, 255:red, 65; green, 117; blue, 5 }  ,draw opacity=1 ][fill={rgb, 255:red, 65; green, 117; blue, 5 }  ,fill opacity=1 ][line width=0.75]      (0, 0) circle [x radius= 1.34, y radius= 1.34]   ;
\draw [shift={(267.93,119.95)}, rotate = 65.72] [color={rgb, 255:red, 65; green, 117; blue, 5 }  ,draw opacity=1 ][fill={rgb, 255:red, 65; green, 117; blue, 5 }  ,fill opacity=1 ][line width=0.75]      (0, 0) circle [x radius= 1.34, y radius= 1.34]   ;
\draw [color={rgb, 255:red, 65; green, 117; blue, 5 }  ,draw opacity=1 ]   (241.2,118.09) .. controls (232.26,131.39) and (241.46,144.03) .. (252.84,150.11) ;
\draw [shift={(252.84,150.11)}, rotate = 28.13] [color={rgb, 255:red, 65; green, 117; blue, 5 }  ,draw opacity=1 ][fill={rgb, 255:red, 65; green, 117; blue, 5 }  ,fill opacity=1 ][line width=0.75]      (0, 0) circle [x radius= 1.34, y radius= 1.34]   ;
\draw [shift={(241.2,118.09)}, rotate = 123.9] [color={rgb, 255:red, 65; green, 117; blue, 5 }  ,draw opacity=1 ][fill={rgb, 255:red, 65; green, 117; blue, 5 }  ,fill opacity=1 ][line width=0.75]      (0, 0) circle [x radius= 1.34, y radius= 1.34]   ;
\draw [color={rgb, 255:red, 208; green, 2; blue, 27 }  ,draw opacity=1 ][line width=1.5]    (214.33,104.67) .. controls (207.33,116.33) and (192.33,118) .. (188.67,133.67) ;
\draw [shift={(188.67,133.67)}, rotate = 103.17] [color={rgb, 255:red, 208; green, 2; blue, 27 }  ,draw opacity=1 ][fill={rgb, 255:red, 208; green, 2; blue, 27 }  ,fill opacity=1 ][line width=1.5]      (0, 0) circle [x radius= 1.74, y radius= 1.74]   ;
\draw [color={rgb, 255:red, 208; green, 2; blue, 27 }  ,draw opacity=1 ][line width=1.5]    (218,98) .. controls (228.67,83.33) and (239,87.33) .. (255.67,82.67) ;
\draw [shift={(255.67,82.67)}, rotate = 344.36] [color={rgb, 255:red, 208; green, 2; blue, 27 }  ,draw opacity=1 ][fill={rgb, 255:red, 208; green, 2; blue, 27 }  ,fill opacity=1 ][line width=1.5]      (0, 0) circle [x radius= 1.74, y radius= 1.74]   ;
\draw [color={rgb, 255:red, 155; green, 155; blue, 155 }  ,draw opacity=1 ] [dash pattern={on 2.25pt off 1.5pt}]  (398.47,135.75) .. controls (434.53,114.42) and (416.88,165.59) .. (480.95,121.86) ;
\draw [shift={(395.58,137.53)}, rotate = 327.12] [fill={rgb, 255:red, 155; green, 155; blue, 155 }  ,fill opacity=1 ][line width=0.08]  [draw opacity=0] (5.36,-2.57) -- (0,0) -- (5.36,2.57) -- cycle    ;
\draw [color={rgb, 255:red, 155; green, 155; blue, 155 }  ,draw opacity=1 ] [dash pattern={on 2.25pt off 1.5pt}]  (481.63,73.43) .. controls (504.8,96.25) and (486.34,96.49) .. (493.66,123.44) ;
\draw [shift={(479.36,71.25)}, rotate = 43.23] [fill={rgb, 255:red, 155; green, 155; blue, 155 }  ,fill opacity=1 ][line width=0.08]  [draw opacity=0] (5.36,-2.57) -- (0,0) -- (5.36,2.57) -- cycle    ;
\draw [color={rgb, 255:red, 155; green, 155; blue, 155 }  ,draw opacity=1 ] [dash pattern={on 2.25pt off 1.5pt}]  (471.35,114.71) .. controls (482.74,103.6) and (495.44,105.98) .. (504.97,115.24) ;
\draw [color={rgb, 255:red, 155; green, 155; blue, 155 }  ,draw opacity=1 ] [dash pattern={on 2.25pt off 1.5pt}]  (501.27,118.95) .. controls (507.35,132.44) and (503.65,142.5) .. (486.97,144.61) ;
\draw [color={rgb, 255:red, 155; green, 155; blue, 155 }  ,draw opacity=1 ] [dash pattern={on 2.25pt off 1.5pt}]  (474.53,117.09) .. controls (465.6,130.39) and (474.8,143.03) .. (486.18,149.11) ;
\draw [color={rgb, 255:red, 155; green, 155; blue, 155 }  ,draw opacity=1 ][line width=0.75]    (447.67,103.67) .. controls (440.67,115.33) and (425.67,117) .. (422,132.67) ;
\draw [shift={(422,132.67)}, rotate = 103.17] [color={rgb, 255:red, 155; green, 155; blue, 155 }  ,draw opacity=1 ][fill={rgb, 255:red, 155; green, 155; blue, 155 }  ,fill opacity=1 ][line width=0.75]      (0, 0) circle [x radius= 1.34, y radius= 1.34]   ;
\draw [color={rgb, 255:red, 155; green, 155; blue, 155 }  ,draw opacity=1 ][line width=0.75]    (451.33,97) .. controls (462,82.33) and (472.33,86.33) .. (489,81.67) ;
\draw [shift={(489,81.67)}, rotate = 344.36] [color={rgb, 255:red, 155; green, 155; blue, 155 }  ,draw opacity=1 ][fill={rgb, 255:red, 155; green, 155; blue, 155 }  ,fill opacity=1 ][line width=0.75]      (0, 0) circle [x radius= 1.34, y radius= 1.34]   ;
\draw [line width=1.5]    (550.31,86.98) -- (488.69,126.62) ;
\draw [shift={(488.69,126.62)}, rotate = 147.24] [color={rgb, 255:red, 0; green, 0; blue, 0 }  ][fill={rgb, 255:red, 0; green, 0; blue, 0 }  ][line width=1.5]      (0, 0) circle [x radius= 3.48, y radius= 3.48]   ;
\draw [shift={(553.68,84.81)}, rotate = 147.24] [fill={rgb, 255:red, 0; green, 0; blue, 0 }  ][line width=0.08]  [draw opacity=0] (9.29,-4.46) -- (0,0) -- (9.29,4.46) -- cycle    ;
\draw [line width=1.5]    (481.76,193.69) -- (488.69,126.62) ;
\draw [shift={(488.69,126.62)}, rotate = 275.9] [color={rgb, 255:red, 0; green, 0; blue, 0 }  ][fill={rgb, 255:red, 0; green, 0; blue, 0 }  ][line width=1.5]      (0, 0) circle [x radius= 3.48, y radius= 3.48]   ;
\draw [shift={(481.35,197.67)}, rotate = 275.9] [fill={rgb, 255:red, 0; green, 0; blue, 0 }  ][line width=0.08]  [draw opacity=0] (9.29,-4.46) -- (0,0) -- (9.29,4.46) -- cycle    ;
\draw [color={rgb, 255:red, 65; green, 117; blue, 5 }  ,draw opacity=1 ][line width=1.5]    (423.57,82.04) -- (488.69,126.62) ;
\draw [shift={(488.69,126.62)}, rotate = 34.39] [color={rgb, 255:red, 65; green, 117; blue, 5 }  ,draw opacity=1 ][fill={rgb, 255:red, 65; green, 117; blue, 5 }  ,fill opacity=1 ][line width=1.5]      (0, 0) circle [x radius= 3.48, y radius= 3.48]   ;
\draw [shift={(420.27,79.78)}, rotate = 34.39] [fill={rgb, 255:red, 65; green, 117; blue, 5 }  ,fill opacity=1 ][line width=0.08]  [draw opacity=0] (9.29,-4.46) -- (0,0) -- (9.29,4.46) -- cycle    ;
\draw [color={rgb, 255:red, 65; green, 117; blue, 5 }  ,draw opacity=1 ][line width=1.5]    (560.01,153.89) .. controls (518.72,129.83) and (530.44,151.14) .. (487.5,135.35) ;
\draw [shift={(487.5,135.35)}, rotate = 200.18] [color={rgb, 255:red, 65; green, 117; blue, 5 }  ,draw opacity=1 ][fill={rgb, 255:red, 65; green, 117; blue, 5 }  ,fill opacity=1 ][line width=1.5]      (0, 0) circle [x radius= 2.61, y radius= 2.61]   ;
\draw [shift={(564,156.26)}, rotate = 211.1] [fill={rgb, 255:red, 65; green, 117; blue, 5 }  ,fill opacity=1 ][line width=0.08]  [draw opacity=0] (6.97,-3.35) -- (0,0) -- (6.97,3.35) -- cycle    ;
\draw [color={rgb, 255:red, 144; green, 19; blue, 254 }  ,draw opacity=1 ][line width=1.5]    (451.33,97) -- (454,94) -- (457.25,90.75) -- (462.5,87) -- (468.75,84.5) -- (475,84) -- (482.25,83.25) -- (489,81.67) -- (493,87.75) -- (494.25,93.25) -- (493.25,99.5) -- (491.75,107.75) -- (495.25,108.5) -- (500.5,111.5) -- (504.97,115.24) ;
\draw [shift={(504.97,115.24)}, rotate = 39.92] [color={rgb, 255:red, 144; green, 19; blue, 254 }  ,draw opacity=1 ][fill={rgb, 255:red, 144; green, 19; blue, 254 }  ,fill opacity=1 ][line width=1.5]      (0, 0) circle [x radius= 2.61, y radius= 2.61]   ;
\draw [color={rgb, 255:red, 144; green, 19; blue, 254 }  ,draw opacity=1 ][line width=1.5]    (447.67,103.67) -- (444,108.5) -- (438.5,113.25) -- (432.75,116.75) -- (428.25,121) -- (425,125.25) -- (422.75,129.5) -- (422,132.67) -- (429.5,137.25) -- (436,139.75) -- (441.75,140.25) -- (449.25,138.75) -- (455.25,136.5) -- (461.25,133.75) -- (468.25,130) -- (470.75,128.5) -- (473,136.75) -- (475.75,141) -- (480.5,145.5) -- (486.18,149.11) ;
\draw [shift={(486.18,149.11)}, rotate = 32.46] [color={rgb, 255:red, 144; green, 19; blue, 254 }  ,draw opacity=1 ][fill={rgb, 255:red, 144; green, 19; blue, 254 }  ,fill opacity=1 ][line width=1.5]      (0, 0) circle [x radius= 2.61, y radius= 2.61]   ;

\draw (449,67.15) node [anchor=north west][inner sep=0.75pt]  [color={rgb, 255:red, 144; green, 19; blue, 254 }  ,opacity=1 ]  {$q$};
\draw (186.5,103.4) node [anchor=north west][inner sep=0.75pt]  [color={rgb, 255:red, 208; green, 2; blue, 27 }  ,opacity=1 ]  {$p$};
\draw (345.75,117.65) node [anchor=north west][inner sep=0.75pt]    {$\Longrightarrow $};

\end{tikzpicture}

%% file: figures/graph-delta.tex
\tikzset{every picture/.style={line width=0.75pt}} 

\begin{tikzpicture}[x=0.75pt,y=0.75pt,yscale=-1,xscale=1]

\draw   (402.57,135.22) -- (372.65,187.03) -- (312.83,187.03) -- (282.92,135.22) -- (312.83,83.41) -- (372.65,83.41) -- cycle ;
\draw  [color={rgb, 255:red, 208; green, 2; blue, 27 }  ,draw opacity=1 ][fill={rgb, 255:red, 255; green, 205; blue, 210 }  ,fill opacity=1 ] (271.7,135.22) .. controls (271.7,129.03) and (276.72,124.01) .. (282.92,124.01) .. controls (289.11,124.01) and (294.13,129.03) .. (294.13,135.22) .. controls (294.13,141.42) and (289.11,146.44) .. (282.92,146.44) .. controls (276.72,146.44) and (271.7,141.42) .. (271.7,135.22) -- cycle ;
\draw  [color={rgb, 255:red, 208; green, 2; blue, 27 }  ,draw opacity=1 ][fill={rgb, 255:red, 255; green, 205; blue, 210 }  ,fill opacity=1 ] (361.44,187.03) .. controls (361.44,180.84) and (366.46,175.82) .. (372.65,175.82) .. controls (378.85,175.82) and (383.87,180.84) .. (383.87,187.03) .. controls (383.87,193.23) and (378.85,198.25) .. (372.65,198.25) .. controls (366.46,198.25) and (361.44,193.23) .. (361.44,187.03) -- cycle ;
\draw  [color={rgb, 255:red, 208; green, 2; blue, 27 }  ,draw opacity=1 ][fill={rgb, 255:red, 255; green, 205; blue, 210 }  ,fill opacity=1 ] (361.44,83.41) .. controls (361.44,77.22) and (366.46,72.2) .. (372.65,72.2) .. controls (378.85,72.2) and (383.87,77.22) .. (383.87,83.41) .. controls (383.87,89.61) and (378.85,94.63) .. (372.65,94.63) .. controls (366.46,94.63) and (361.44,89.61) .. (361.44,83.41) -- cycle ;
\draw  [color={rgb, 255:red, 65; green, 117; blue, 5 }  ,draw opacity=1 ][fill={rgb, 255:red, 227; green, 255; blue, 200 }  ,fill opacity=1 ] (301.61,83.41) .. controls (301.61,77.22) and (306.63,72.2) .. (312.83,72.2) .. controls (319.02,72.2) and (324.05,77.22) .. (324.05,83.41) .. controls (324.05,89.61) and (319.02,94.63) .. (312.83,94.63) .. controls (306.63,94.63) and (301.61,89.61) .. (301.61,83.41) -- cycle ;
\draw  [color={rgb, 255:red, 65; green, 117; blue, 5 }  ,draw opacity=1 ][fill={rgb, 255:red, 227; green, 255; blue, 200 }  ,fill opacity=1 ] (391.35,135.22) .. controls (391.35,129.03) and (396.37,124.01) .. (402.57,124.01) .. controls (408.76,124.01) and (413.78,129.03) .. (413.78,135.22) .. controls (413.78,141.42) and (408.76,146.44) .. (402.57,146.44) .. controls (396.37,146.44) and (391.35,141.42) .. (391.35,135.22) -- cycle ;
\draw  [color={rgb, 255:red, 65; green, 117; blue, 5 }  ,draw opacity=1 ][fill={rgb, 255:red, 227; green, 255; blue, 200 }  ,fill opacity=1 ] (301.61,187.03) .. controls (301.61,180.84) and (306.63,175.82) .. (312.83,175.82) .. controls (319.02,175.82) and (324.05,180.84) .. (324.05,187.03) .. controls (324.05,193.23) and (319.02,198.25) .. (312.83,198.25) .. controls (306.63,198.25) and (301.61,193.23) .. (301.61,187.03) -- cycle ;

\draw (273.92,130.16) node [anchor=north west][inner sep=0.75pt]  [color={rgb, 255:red, 208; green, 2; blue, 27 }  ,opacity=1 ]  {$\gamma _{1}$};
\draw (363.65,181.68) node [anchor=north west][inner sep=0.75pt]  [color={rgb, 255:red, 208; green, 2; blue, 27 }  ,opacity=1 ]  {$\gamma _{2}$};
\draw (363.99,78.32) node [anchor=north west][inner sep=0.75pt]  [color={rgb, 255:red, 208; green, 2; blue, 27 }  ,opacity=1 ]  {$\gamma _{3}$};
\draw (392.65,130.16) node [anchor=north west][inner sep=0.75pt]  [color={rgb, 255:red, 65; green, 117; blue, 5 }  ,opacity=1 ]  {$w_{1}$};
\draw (302.25,182.01) node [anchor=north west][inner sep=0.75pt]  [color={rgb, 255:red, 65; green, 117; blue, 5 }  ,opacity=1 ]  {$w_{3}$};
\draw (302.58,77.65) node [anchor=north west][inner sep=0.75pt]  [color={rgb, 255:red, 65; green, 117; blue, 5 }  ,opacity=1 ]  {$w_{2}$};

\end{tikzpicture}

%% file: figures/final-proof-setup.tex
\tikzset{every picture/.style={line width=0.75pt}} 

\begin{tikzpicture}[x=0.75pt,y=0.75pt,yscale=-1,xscale=1]

\draw [line width=1.5]    (52.5,146.25) -- (510,146.25) ;
\draw [shift={(514,146.25)}, rotate = 180] [fill={rgb, 255:red, 0; green, 0; blue, 0 }  ][line width=0.08]  [draw opacity=0] (9.29,-4.46) -- (0,0) -- (9.29,4.46) -- cycle    ;
\draw [shift={(48.5,146.25)}, rotate = 0] [fill={rgb, 255:red, 0; green, 0; blue, 0 }  ][line width=0.08]  [draw opacity=0] (9.29,-4.46) -- (0,0) -- (9.29,4.46) -- cycle    ;
\draw [color={rgb, 255:red, 74; green, 144; blue, 226 }  ,draw opacity=1 ][line width=3]    (396.5,146) .. controls (396.31,140.63) and (383.5,138.5) .. (393.5,131) .. controls (403.5,123.5) and (468.25,128.5) .. (421.75,94.75) .. controls (375.25,61) and (277.5,41.25) .. (250.5,61.75) .. controls (223.5,82.25) and (233.75,105.75) .. (207,98.75) .. controls (180.25,91.75) and (171,140.75) .. (170.75,146) ;
\draw [shift={(170.75,146)}, rotate = 92.73] [color={rgb, 255:red, 74; green, 144; blue, 226 }  ,draw opacity=1 ][fill={rgb, 255:red, 74; green, 144; blue, 226 }  ,fill opacity=1 ][line width=3]      (0, 0) circle [x radius= 2.55, y radius= 2.55]   ;
\draw [shift={(396.5,146)}, rotate = 268] [color={rgb, 255:red, 74; green, 144; blue, 226 }  ,draw opacity=1 ][fill={rgb, 255:red, 74; green, 144; blue, 226 }  ,fill opacity=1 ][line width=3]      (0, 0) circle [x radius= 2.55, y radius= 2.55]   ;
\draw [color={rgb, 255:red, 65; green, 117; blue, 5 }  ,draw opacity=1 ]   (297.13,146.38) -- (288,157.5) ;
\draw [color={rgb, 255:red, 65; green, 117; blue, 5 }  ,draw opacity=1 ]   (301.75,136.25) -- (297.13,146.38) ;
\draw    (354.25,146) .. controls (372.75,146) and (364.5,165.5) .. (378.5,160.5) .. controls (392.5,155.5) and (396.69,151.38) .. (396.5,146) .. controls (396.31,140.63) and (383.5,138.5) .. (393.5,131) .. controls (403.5,123.5) and (468.25,128.5) .. (421.75,94.75) .. controls (375.25,61) and (277.5,41.25) .. (250.5,61.75) .. controls (223.5,82.25) and (233.75,105.75) .. (207,98.75) .. controls (180.25,91.75) and (171,140.75) .. (170.75,146) .. controls (170.5,151.25) and (160,176.25) .. (146.5,176.25) .. controls (133,176.25) and (137.5,113) .. (118.5,113.5) .. controls (99.5,114) and (86.75,131) .. (88.5,147.25) .. controls (90.25,163.5) and (151.63,227.88) .. (161.75,228.25) .. controls (171.88,228.63) and (170.25,167.5) .. (195.25,142.5) .. controls (220.25,117.5) and (210.5,188) .. (222.5,185.75) .. controls (234.5,183.5) and (231.25,147) .. (242.75,146.25) ;
\draw  [color={rgb, 255:red, 155; green, 155; blue, 155 }  ,draw opacity=1 ][dash pattern={on 1.5pt off 2.25pt}] (276,146.38) .. controls (276,134.71) and (285.46,125.25) .. (297.13,125.25) .. controls (308.79,125.25) and (318.25,134.71) .. (318.25,146.38) .. controls (318.25,158.04) and (308.79,167.5) .. (297.13,167.5) .. controls (285.46,167.5) and (276,158.04) .. (276,146.38) -- cycle ;
\draw [color={rgb, 255:red, 208; green, 2; blue, 27 }  ,draw opacity=1 ]   (297.13,146.38) ;
\draw [shift={(297.13,146.38)}, rotate = 0] [color={rgb, 255:red, 208; green, 2; blue, 27 }  ,draw opacity=1 ][fill={rgb, 255:red, 208; green, 2; blue, 27 }  ,fill opacity=1 ][line width=0.75]      (0, 0) circle [x radius= 3.35, y radius= 3.35]   ;
\draw [color={rgb, 255:red, 65; green, 117; blue, 5 }  ,draw opacity=1 ]   (83.09,217.54) .. controls (111.77,121.67) and (248.4,191.66) .. (288,157.5) ;
\draw [shift={(288,157.5)}, rotate = 319.22] [color={rgb, 255:red, 65; green, 117; blue, 5 }  ,draw opacity=1 ][fill={rgb, 255:red, 65; green, 117; blue, 5 }  ,fill opacity=1 ][line width=0.75]      (0, 0) circle [x radius= 2.01, y radius= 2.01]   ;
\draw [shift={(82.25,220.5)}, rotate = 284.94] [fill={rgb, 255:red, 65; green, 117; blue, 5 }  ,fill opacity=1 ][line width=0.08]  [draw opacity=0] (5.36,-2.57) -- (0,0) -- (5.36,2.57) -- cycle    ;
\draw [color={rgb, 255:red, 65; green, 117; blue, 5 }  ,draw opacity=1 ]   (377.97,50.5) .. controls (362.84,91.72) and (282.98,91.64) .. (301.75,136.25) ;
\draw [shift={(301.75,136.25)}, rotate = 67.18] [color={rgb, 255:red, 65; green, 117; blue, 5 }  ,draw opacity=1 ][fill={rgb, 255:red, 65; green, 117; blue, 5 }  ,fill opacity=1 ][line width=0.75]      (0, 0) circle [x radius= 2.01, y radius= 2.01]   ;
\draw [shift={(379,47.25)}, rotate = 105.09] [fill={rgb, 255:red, 65; green, 117; blue, 5 }  ,fill opacity=1 ][line width=0.08]  [draw opacity=0] (5.36,-2.57) -- (0,0) -- (5.36,2.57) -- cycle    ;
\draw  [color={rgb, 255:red, 155; green, 155; blue, 155 }  ,draw opacity=1 ][dash pattern={on 1.5pt off 2.25pt}] (251.22,146.38) .. controls (251.22,121.02) and (271.77,100.47) .. (297.13,100.47) .. controls (322.48,100.47) and (343.03,121.02) .. (343.03,146.38) .. controls (343.03,171.73) and (322.48,192.28) .. (297.13,192.28) .. controls (271.77,192.28) and (251.22,171.73) .. (251.22,146.38) -- cycle ;
\draw [color={rgb, 255:red, 74; green, 144; blue, 226 }  ,draw opacity=1 ]   (266,145.75) .. controls (281.25,98.5) and (324.75,113.25) .. (330.75,146) ;
\draw [shift={(330.75,146)}, rotate = 79.62] [color={rgb, 255:red, 74; green, 144; blue, 226 }  ,draw opacity=1 ][fill={rgb, 255:red, 74; green, 144; blue, 226 }  ,fill opacity=1 ][line width=0.75]      (0, 0) circle [x radius= 2.01, y radius= 2.01]   ;
\draw [shift={(266,145.75)}, rotate = 287.89] [color={rgb, 255:red, 74; green, 144; blue, 226 }  ,draw opacity=1 ][fill={rgb, 255:red, 74; green, 144; blue, 226 }  ,fill opacity=1 ][line width=0.75]      (0, 0) circle [x radius= 2.01, y radius= 2.01]   ;
\draw [color={rgb, 255:red, 65; green, 117; blue, 5 }  ,draw opacity=1 ]   (300.75,116) ;
\draw [shift={(300.75,116)}, rotate = 0] [color={rgb, 255:red, 65; green, 117; blue, 5 }  ,draw opacity=1 ][fill={rgb, 255:red, 65; green, 117; blue, 5 }  ,fill opacity=1 ][line width=0.75]      (0, 0) circle [x radius= 2.01, y radius= 2.01]   ;

\draw (68,152.15) node [anchor=north west][inner sep=0.75pt]    {$\rho $};
\draw (211.5,61.9) node [anchor=north west][inner sep=0.75pt]    {$C _{n}$};
\draw (385.75,43.15) node [anchor=north west][inner sep=0.75pt]  [color={rgb, 255:red, 65; green, 117; blue, 5 }  ,opacity=1 ]  {$w_{1}$};
\draw (91,202.9) node [anchor=north west][inner sep=0.75pt]  [color={rgb, 255:red, 65; green, 117; blue, 5 }  ,opacity=1 ]  {$w_{2}$};
\draw (322,119.4) node [anchor=north west][inner sep=0.75pt]  [font=\footnotesize,color={rgb, 255:red, 74; green, 144; blue, 226 }  ,opacity=1 ]  {$a_{1}$};
\draw (408.25,104.9) node [anchor=north west][inner sep=0.75pt]  [font=\footnotesize,color={rgb, 255:red, 74; green, 144; blue, 226 }  ,opacity=1 ]  {$a_{n}$};

\end{tikzpicture}